\documentclass[10pt]{article}
\usepackage{epsfig}
\usepackage{amsmath}
\usepackage{amssymb}
\usepackage{graphicx}
\usepackage[latin1]{inputenc}
\usepackage{amsthm}
\usepackage[french,english]{babel}
{\catcode `\@=11 \global\let\AddToReset=\@addtoreset}
\AddToReset{equation}{section}

\newtheorem{lemma}{\bf Lemma}[section]

\newtheorem{property}{Property}[section]

\newtheorem{@definition}{\sc Definition}[section]

\newtheorem{@remark}{\sc Remark}[section]

\newtheorem{@example}{\sc Example}[section]

\newcommand{\beqn}{\begin{displaymath}}
\newcommand{\eeqn}{\end{displaymath}}
\newcommand{\beq}{\begin{equation}}  % numbered (single equation)
\newcommand{\eeq}{\end{equation}}

\def\mathsf{\bf}
\def\N{\mathbb{N}}

\def\R{\mathbb{R}}
\def\Z{\mathbb{Z}}

\def\E{\mathrm E}

\def\text{\mbox}

\def\1{{\bf 1}}

\newcommand{\Cov}{\mbox{Cov}}

\newcommand{\Loi}{\mathcal{L}}

\def\limiteloiN{\renewcommand{\arraystretch}{0.5}
\begin{array}[t]{c}
\stackrel{{\Loi}}{\longrightarrow} \\
{\scriptstyle N\rightarrow\infty}
\end{array}\renewcommand{\arraystretch}{1}}

\def\limiteloiNm{\renewcommand{\arraystretch}{0.5}
\begin{array}[t]{c}
\stackrel{{\Loi}}{\longrightarrow} \\
{\scriptstyle [N/m]\wedge m \rightarrow\infty}
\end{array}\renewcommand{\arraystretch}{1}}

\def\limiteprobaN{\renewcommand{\arraystretch}{0.5}
\begin{array}[t]{c}
\stackrel{{\cal P}}{\longrightarrow} \\
{\scriptstyle N \rightarrow\infty}
\end{array}\renewcommand{\arraystretch}{1}}

\def\limiteN{\renewcommand{\arraystretch}{0.5}
\begin{array}[t]{c}
\stackrel{}{\longrightarrow} \\
{\scriptstyle N\rightarrow\infty}
\end{array}\renewcommand{\arraystretch}{1}}

\def\limitet{\renewcommand{\arraystretch}{0.5}
\begin{array}[t]{c}
\stackrel{}{\longrightarrow} \\
{\scriptstyle t\rightarrow\infty}
\end{array}\renewcommand{\arraystretch}{1}}

\def\limitet0{\renewcommand{\arraystretch}{0.5}
\begin{array}[t]{c}
\stackrel{}{\longrightarrow} \\
{\scriptstyle t\rightarrow 0}
\end{array}\renewcommand{\arraystretch}{1}}

\newtheorem{thm}{Theorem}
\newtheorem{rem}{Remark}

\newtheorem{prop}{Proposition}

\newtheorem{popy}{Property}

\def\Cov{\mathrm{Cov}}

\begin{document}

\title{\bf Semiparametric stationarity tests based on  adaptive multidimensional increment ratio statistics}

\author{\centerline{Jean-Marc Bardet and Béchir Dola} \\
\small {\tt bardet@univ-paris1.fr},~~ \small {\tt bechir.dola@malix.univ-paris1.fr}\\
~\\
SAMM, Université Panthéon-Sorbonne (Paris I), 90 rue de Tolbiac, 75013 Paris, FRANCE}

\maketitle
\begin{abstract}
In this paper, we show that the  adaptive multidimensional increment ratio estimator of the long range memory parameter defined in Bardet and Dola (2012) satisfies a central limit theorem (CLT in the sequel) for a large semiparametric class of Gaussian fractionally integrated  processes with memory parameter $d \in (-0.5,1.25)$. Since the asymptotic variance of this CLT can be computed, tests of stationarity or nonstationarity distinguishing the assumptions $d<0.5$ and $d \geq 0.5$ are constructed. These tests are also consistent tests of unit root. Simulations done on a large benchmark of short memory, long memory and non stationary processes show the accuracy of the tests with respect to other usual stationarity or nonstationarity tests (LMC, V/S, ADF and PP tests). Finally, the estimator and tests are applied to log-returns of famous economic data and to their absolute value power laws.
\end{abstract}
\begin{quote}
\end{quote}

\vskip1cm

\section{Introduction}

Consider the set $I(d)$ of fractionally integrated time series $X=(X_k)_{k\in \Z}$ for $-0.5<d<1.5$ by:\\
~\\
{\bf Assumption $I(d)$:} {\em $X=(X_t)_{t\in \Z}$  is a time series if there exists a continuous function $f^*:[-\pi,\pi] \to [0,\infty[$ satisfying:}
\begin{enumerate}
\item  {\em  if $-0.5<d<0.5$,  $X$ is a stationary process having a spectral density $f$ satisfying}
\begin{eqnarray}\label{IdS}
f(\lambda) = |\lambda|^{-2d}f^*(\lambda)\quad\mbox{for all $\lambda \in (-\pi,0)\cup (0,\pi)$, with $f^*(0)>0$}.
\end{eqnarray}
\item  {\em  if $0.5\leq d<1.5$, $U=(U_t)_{t\in \Z}=X_{t}-X_{t-1}$ is a stationary process having a spectral density $f$ satisfying}
\begin{eqnarray}\label{IdAS}
f(\lambda) =  |\lambda|^{2-2d}f^*(\lambda)\quad\mbox{for all $\lambda \in (-\pi,0)\cup (0,\pi)$, with $f^*(0)>0$}.
\end{eqnarray}
\end{enumerate}

The case $d\in (0,0.5)$ is the case of long-memory processes, while short-memory processes are considered when $-0.5< d\leq 0$ and nonstationary processes when $d\geq 0.5$. ARFIMA$(p,d,q)$ processes (which are linear processes) or fractional Gaussian noises (with parameter $H=d+1/2\in (0,1	)$) are famous examples of processes satisfying Assumption $I(d)$. The purpose of this paper is twofold: firstly, we establish the consistency of an adaptive semiparametric estimator of $d$ for any $d\in (-0.5,1.25)$. Secondly, we use this estimator for building new semiparametric stationary tests.  \\
Numerous articles have been devoted to estimate $d$ in the case $d\in (-0.5,0.5)$. The books of Beran (1994) or  Doukhan {\it et al.} (2003) provide large surveys of such parametric (mainly maximum likelihood or Whittle estimators) or semiparametric estimators (mainly local Whittle, log-periodogram or wavelet based estimators). Here we will restrict our discussion to the case of semiparametric estimators that are best suited to address the general case of processes satisfying Assumption $I(d)$. Even if first versions of local Whittle, log-periodogramm and wavelet based estimators (see for instance Robinson, 1995a and 1995b, Abry and Veitch, 1998) are only considered in the case $d<0.5$, new extensions have been provided for also estimating $d$ when $d\geq 0.5$ (see for instance Hurvich and Ray, 1995, Velasco, 1999a,  Velasco and Robinson, 2000,
Moulines and Soulier, 2003,  Shimotsu and Phillips, 2005, Giraitis {\it et al.}, 2003, 2006, Abadir {\it et al.}, 2007 or Moulines {\it et al.}, 2007). Moreover, adaptive versions of these estimators have also been defined for avoiding any trimming or bandwidth parameters generally required by these methods (see for instance Giraitis {\it et al.}, 2000, Moulines and Soulier, 2003, or Veitch {\it et al.}, 2003, or Bardet {\it et al.}, 2008). However there still no exists an adaptive estimator of $d$ satisfying a central limit theorem (for providing confidence intervals or tests) and valid for $d<0.5$ but also for $d\geq 0.5$. This is the first objective of this paper and it will be achieved using multidimensional Increment Ratio (IR) statistics. \\
Indeed, Surgailis {\it et al.} (2008) first defined the statistic $IR_N$ (see its definition in \eqref{defIR}) from an observed trajectory $(X_1,\ldots,X_N)$. Its asymptotic behavior is studied and a central limit theorem (CLT in the sequel) is established for $d \in (-0.5,0.5)\cup (0.5,1.25)$ inducing a CLT. Therefore, the estimator $\widehat d_N=\Lambda_0^{-1}(IR_N)$, where $d \mapsto \Lambda_0(d)$ is  a smooth and increasing function, is a consistent estimator of $d$ satisfying also a CLT (see more details below). However this new estimator was not totally satisfying: firstly, it requires the knowledge of the second order behavior of the spectral density that is clearly unknown in practice. Secondly, its numerical accuracy is interesting but clearly less than the one of local Whittle or log-periodogram estimators. As a consequence, in Bardet and Dola (2012), we built an adaptive multidimensional $IR$ estimator $\widetilde d_N^{IR}$ (see its definition in \eqref{dtilde}) answering to both these points 
but only for
$-0.5<d<0.5$. This is an adaptive semiparametric estimator of $d$ and its numerical performances are often better than the ones of local Whittle or log-periodogram estimators.  \\
Here we extend this preliminary work to the case $0.5\leq d <1.25$. Hence we obtain a CLT satisfied by $\widetilde d_N^{IR}$ for all $d\in (-0.5,1.25)$ with an explicit asymptotic variance depending only on $d$ and this notably allows to obtain confidence intervals. The case $d=0.5$ is now studied and this offers new interesting perspectives: our adaptive estimator can be used for building a stationarity (or nonstationarity) test since $0.5$ is the ``border number'' between stationarity and nonstationarity.  ~\\
\\
There exist several famous stationarity (or nonstationarity) tests. For stationarity tests we may cite the KPSS (Kwiotowski, Phillips, Schmidt, Shin) test (see for instance Hamilton, 1994, p. 514) and LMC test (see Leybourne and  McCabe, 2000). For nonstationarity tests we may cite the Augmented Dickey-Fuller test (ADF test in the sequel, see Hamilton, 1994, p. 516-528) and the Philipps and Perron test (PP test in the sequel, see for instance Elder, 2001, p. 137-146). All these tests are unit root tests, {\it i.e.} and roughly speaking, semiparametric tests based on the model $X_t=\rho\, X_{t-1} + \varepsilon_t$ with $|\rho|\leq 1$. A test about $d=0.5$ for a process satisfying Assumption $I(d)$ is therefore a refinement of a basic unit root test since the case $\rho=1$ is a particular case of $I(1)$ and the case $|\rho|<1$ a particular case of $I(0)$. Thus, a stationarity (or nonstationarity test) based on the
estimator of $d$ provides  a more sensible test than usual unit root tests. \\
This principle of stationarity test linked to $d$ was also already investigated in many articles. We can notably cite Robinson (1994), Tanaka (1999), Ling and Li (2001), Ling (2003) or Nielsen (2004). However,  all these papers provide parametric tests, with a specified model (for instance ARFIMA or ARFIMA-GARCH processes). More recently, several papers have been devoted to the construction of semi-parametric tests, see for in instance Giraitis {\it et al.} (2006), Abadir {\it et al.} (2007) or Surgailis {\it et al.} (2006).  \\
Here we slightly restrict the general class $I(d)$ to the Gaussian semiparametric class $IG(d,\beta)$ defined below (see the beginning of Section \ref{MIR}). For processes belonging to this class, we construct a new  stationarity test $\widetilde S_N$ which accepts the stationarity assumption when $\widetilde d_N^{IR}\leq 0.5+s$ with $s$ a threshold depending on the type I error test and $N$, while the new nonstationarity test $\widetilde T_N$ accepts the nonstationarity assumption when $\widetilde d_N^{IR}\geq 0.5-s$. Note that $\widetilde d_N^{IR}\leq s'$ also provides a test for deciding between short and long range dependency, as this is done by the V/S test (see details in Giraitis {\it et al.}, 2003)\\
~\\
In Section \ref{simu}, numerous simulations are realized on several models of time series (short and long memory processes). \\
First, the new multidimensional IR estimator  $\widetilde d_N^{IR}$ is compared to the most efficient and famous semiparametric estimators for $d \in [-0.4,1.2]$; the performances of $\widetilde d_N^{IR}$  are convincing  and equivalent to close to other adaptive estimators (except for extended local Whittle estimator defined in Abadir {\it et al.}, 2007, which provides the best results but is not an adative estimator). \\
Secondly, the new stationarity $\widetilde S_N$ and nonstationarity $\widetilde T_N$ tests are compared on the same benchmark of processes to the most famous unit root tests (LMC, V/S, ADF and PP tests). And the results are quite surprising: even on AR$[1]$ or ARIMA$[1,1,0]$ processes, multidimensional IR $\widetilde S_N$ and $\widetilde T_N$ tests provide convincing results as well as tests built from the extended local Whittle estimator. Note however that ADF and PP tests provide results slightly better than these tests for these processes. For long-memory processes (such as ARFIMA processes), the results are clear: $\widetilde S_N$ and $\widetilde T_N$ tests are efficient tests of (non)stationarity while LMC, ADF and PP tests are not relevant at all. \\
Finally, we studied the stationarity and long range dependency properties of Econometric data. We chose to apply estimators and tests to the log-returns of daily closing value of $5$ classical Stocks and Exchange Rate Markets. After cutting the series in $3$ stages using an algorithm of change detection, we found again this well known result: the log-returns are stationary and short memory processes while absolute values or powers of absolute values of log-returns are generally stationary and long memory processes. Classical stationarity or nonstationarity tests are not able to lead to such conclusions. We also remarked that these time series during the ``last'' (and third)  stages (after 1997 for almost all) are generally closer to nonstationary processes than during the previous stages with a long memory parameter close to $0.5$.  \\
~\\
The forthcoming Section~\ref{MIR} is devoted to the definition and asymptotic behavior of the adaptive multidimensional IR estimator of $d$. The stationarity and nonstationarity tests are presented in Section \ref{test} while Section \ref{simu} provides the results of simulations and application on econometric data. Finally Section \ref{proofs} contains the proofs of main results.

\section{The multidimensional increment ratio statistic}\label{MIR}
In this paper we consider a semiparametric class $IG(d,\beta)$: for $0\leq d<1.5$ and $\beta>0$ define: \\
~\\
{\bf Assumption $IG(d,\beta)$:~} {\em $X=(X_t)_{t\in \Z}$  is a Gaussian time series such that there exist $\epsilon >
0$, $c_0>0$, $c'_0>0$  and $c_1 \in \R$ satisfying:}
\begin{enumerate}
\item  {\em  if $d<0.5$,  $X$ is a stationary process having a spectral density $f$ satisfying for all $\lambda \in (-\pi,0)\cup (0,\pi)$}
\begin{eqnarray}\label{AssumptionS}
f(\lambda) =  c_{0}|\lambda|^{-2d}+c_{1}|\lambda|^{-2d+\beta} +O\big(|\lambda|^{-2d+\beta+\epsilon}\big) \quad \mbox{and}\quad |f'(\lambda)|\leq  c'_0 \, \lambda^{-2d-1}.
\end{eqnarray}
\item  {\em  if $0.5\leq d<1.5$, $U=(U_t)_{t\in \Z}=X_{t}-X_{t-1}$ is a stationary process having a spectral density $f$ satisfying for all $\lambda \in (-\pi,0)\cup (0,\pi)$}
\begin{eqnarray}\label{AssumptionAS}
f(\lambda) =  c_{0}|\lambda|^{2-2d}+c_{1}|\lambda|^{2-2d+\beta} +O\big(|\lambda|^{2-2d+\beta+\epsilon}\big) \quad \mbox{and}\quad |f'(\lambda)|\leq  c'_0 \, \lambda^{-2d+1}.
\end{eqnarray}
\end{enumerate}
Note that  Assumption $IG(d,\beta)$ is a particular (but still general) case of the more usual set $I(d)$ of fractionally integrated processes defined above.
\begin{rem}\label{lin}
We considered here only Gaussian processes. In Surgailis {\it et al.} (2008) and Bardet and Dola (2012), simulations exhibited that the obtained limit theorems should be also valid for linear processes. However a theoretical proof of such result would require limit theorems for functionals of multidimensional linear processes difficult to be established.
\end{rem}
~\\
In this section, under Assumption $IG(d,\beta)$, we establish central limit theorems which extend to the case $d\in [0.5,1.25)$ those already obtained in Bardet and Dola (2012) for $d\in (-0.5,0.5)$. \\
Let $X=(X_k)_{k\in \N}$ be a process satisfying Assumption $IG(d,\beta)$ and $(X_1,\cdots,X_N)$ be a path of $X$. For any $\ell \in \N^*$ define
\begin{eqnarray}\label{defIR}
IR_N(\ell):=\frac{1}{N-3\ell} \,  \sum_{k=0}^{N-3\ell-1}\frac{\displaystyle \Big |(\sum_{t=k+1}^{k+\ell}X_{t+\ell}-\sum_{t=k+1}^{k+\ell}X_{t})+(\sum_{t=k+\ell+1}^{k+2\ell}X_{t+\ell}-\sum_{t=k+\ell+1}^{k+2\ell}X_{t})\Big |}{\displaystyle \Big |(\sum_{t=k+1}^{k+\ell}X_{t+\ell}-\sum_{t=k+1}^{k+\ell} X_{t})\Big|+\Big|(\sum_{t=k+\ell+1}^{k+2\ell} X_{t+\ell}-\sum_{t=k+\ell+1}^{k+2\ell}X_{t})\Big|}.
\end{eqnarray}
The statistic $IR_N$ was first defined in Surgailis {\it et al.} (2008) as a way to estimate the memory parameter. In Bardet and Surgailis (2011) a simple version of IR-statistic was also introduced to measure the roughness of continuous time processes. The main interest of such a statistic is to be very robust to additional or multiplicative trends.  \\
As in Bardet and Dola (2012), let $m_{j}=j\, m,~j=1,\cdots,p$ with $p \in \N^*$ and $m\in \N^*$, and define the random vector $(IR_{N}(m_j))_{1\leq j\leq p}$.
In the sequel we naturally extend the results obtained for $m\in \N^*$ to $m\in (0,\infty)$ by the convention: $(IR_{N}(j \, m))_{1\leq j\leq p}=(IR_{N}(j \, [m]))_{1\leq j\leq p}$ (which changes nothing to the asymptotic results). \\
  \\
For $H\in (0,1)$, let $B_H=(B_H(t))_{t\in \R}$ be a standard fractional Brownian motion, {\it i.e.} a centered Gaussian process having stationary increments and such as $\Cov\big (B_H(t)\,,\, B_H(s)\big )=\frac 1 2\, \big (|t|^{2H} +|s|^{2H} -|t-s|^{2H}  \big )$. Now, using obvious modifications of Surgailis {\it et al.} (2008), for $d \in (-0.5,1.25)$ and $p\in \N^*$, define the stationary multidimensional centered Gaussian processes $\big (Z_d^{(1)}(\tau),\cdots, Z_d^{(p)}(\tau)\big)$ such as for $\tau \in \R$,
\begin{eqnarray}\label{DefZ2}
Z_{d}^{(j)}(\tau):=\left \{ \begin{array}{ll}
\displaystyle \frac{\sqrt{2d(2d+1)}} {\sqrt{|4^{d+0.5}-4|}} \,  \int_{0}^{1}\big ( B_{d-0.5}(\tau+s+j)-B_{d-0.5}(\tau+s)\big )ds&\mbox{if $d\in (0.5,1.25)$} \\
\displaystyle \frac{1} {\sqrt{|4^{d+0.5}-4|}} \,  \big ( B_{d+0.5}(\tau+2j)-2\, B_{d+0.5}(\tau+j)+B_{d+0.5}(\tau)\big )&\mbox{if $d\in (-0.5,05)$}
\end{array} \right .
\end{eqnarray}
and by continuous extension when $d\to 0.5$:
\begin{eqnarray*}
\Cov \big ( Z_{0.5}^{(i)}(0), Z_{0.5}^{(j)}(\tau)\big ):=\frac 1 {4 \, \log 2 } \, \big ( -h(\tau+i-j)+h(\tau+i)+h(\tau-j)-h(\tau)\big )\quad \mbox{for $\tau \in \R$},
\end{eqnarray*}
with $h(x)=\frac 1 2 \big (|x-1|^2\log|x-1|+|x+1|^2\log|x+1|-2|x|^2\log|x|\big )$ for $x\in \R$, using the convention $0 \times \log 0=0$. Now, we establish a multidimensional central limit theorem satisfied by $(IR_{N}(j \, m))_{1\leq j\leq p}$ for all $d\in (-0.5,1.25)$:
\begin{prop}\label{MCLT}
Assume that Assumption $IG(d,\beta)$ holds with $-0.5\leq d<1.25$ and $\beta>0$. Then
\begin{eqnarray}\label{TLC1}
\sqrt{\frac{N}{m}}\Big (IR_N(j \, m)-\E \big [IR_N(j \, m)\big ]\Big )_{1\leq j \leq p}\limiteloiNm {\cal N}(0, \Gamma_p(d))
\end{eqnarray}
with $\Gamma_p(d)=(\sigma_{i,j}(d))_{1\leq i,j\leq p}$ where for $i,j\in \{1,\ldots,p\}$,
\begin{eqnarray}\label{ssigma}
\sigma_{i,j}(d):&=&\int_{-\infty}^{\infty}\Cov \Big (\frac{|Z^{(i)}_{d}(0)+Z^{(i)}_{d}(i)|}{|Z^{(i)}_{d}(0)|+|Z^{(i)}_{d}(i)|} \frac{|Z^{(j)}_{d}(\tau)+Z^{(j)}_{d}(\tau+j)|}{|Z^{(j)}_{d}(\tau)|+|Z^{(j)}_{d}(\tau+j)|}\Big )d\tau.
\end{eqnarray}
\end{prop}
\noindent The proof of this proposition as well as all the other proofs is given in Section \ref{Appendix}. As numerical experiments seem to show, we will assume in the sequel that $\Gamma_p(d)$ is a definite positive matrix for all $d \in (-0.5,1.25)$.\\
Now, this central limit theorem can be used for estimating $d$. To begin with,

\begin{property}\label{devEIR}
Let $X$ satisfying Assumption $IG(d,\beta)$ with $0.5\leq d<1.5$ and $0<\beta\leq 2$. Then, there exists a non-vanishing constant $K(d,\beta)$ depending only on $d$ and $\beta$ such that for $m$ large enough,
\begin{eqnarray*}
\E \big [IR_N(m)\big ] =\left\{ \begin{array}{ll}  \Lambda_0(d)+K(d,\beta)\times m^{-\beta} \, \big (1 + o(1)\big )& \mbox {if}~\beta<1+2d \\
 \Lambda_0(d)+ K(0.5,\beta) \times m^{-2} \log m  \, \big (1 + o(1)\big )& \mbox {if}~\beta=2~\mbox{and}~d=0.5
 \end{array}\right .
\end{eqnarray*}
\begin{eqnarray}\label{DefinitionRhod} %\label{Lambda0d}
\mbox{with}\quad \Lambda_0(d)&:=&\Lambda(\rho(d))\quad \mbox{where} \quad \rho(d):=\left \{\begin{array}{ll}\displaystyle \frac{4^{d+1.5}-9^{d+0.5}-7}{2(4-4^{d+0.5})}& \mbox{for}\quad 0.5<d<1.5\\
\displaystyle \frac{9\log(3)}{8\log(2)}-2 & \mbox{for}\quad d=0.5 \end{array} \right. \\
\mbox{and} && \Lambda(r):=\frac{2}{\pi}\, \arctan\sqrt{\frac{1+r}{1-r}}+\frac{1}{\pi}\, \sqrt{\frac{1+r}{1-r}}\log(\frac{2}{1+r})\quad \mbox{for $|r|\leq 1$}.
\end{eqnarray}
\end{property}
\noindent Therefore by choosing $m$ and $N$ such as  $\big (\sqrt {N/m} \big )m^{-\beta}\log m \to 0$ when $m,N\to
\infty$, the term $\E \big [IR(jm)\big ]$ can be replaced by $\Lambda_0(d)$ in Proposition \ref{MCLT}. Then, using the Delta-method with the function $(x_i)_{1\leq i \leq p} \mapsto (\Lambda^{-1}_0(x_i))_{1\leq i \leq p}$ (the function $d \in (-0.5,1.5) \to \Lambda_0(d)$ is a ${\cal C}^\infty$ increasing function), we obtain:
\begin{thm}\label{cltnada}
Let $\widehat d_N(j \, m):=\Lambda_0^{-1}\big (IR_N(j \, m)\big )$ for $1\leq j \leq p$. Assume that Assumption $IG(d,\beta)$ holds with $0.5\leq d<1.25$ and $0<\beta\leq 2$. Then if $m \sim C\, N^\alpha$ with $C>0$ and $(1+2\beta)^{-1}<\alpha<1$,
\begin{eqnarray}\label{cltd}
\sqrt{\frac{N}{m}}\Big (\widehat d_N(j \, m)-d\Big )_{1\leq j \leq p}\limiteloiN {\cal N}\Big(0,(\Lambda'_0(d))^{-2}\, \Gamma_p(d)\Big ).
\end{eqnarray}
\end{thm}
\noindent This result is an extension to the case $0.5\leq d \leq 1.25$ from the case $-0.5<d<0.5$ already obtained in Bardet and Dola (2012). Note that the consistency of $\widehat d_N(j \, m)$ is ensured when $1.25\leq d <1.5$ but the previous CLT does not hold (the asymptotic variance of $\sqrt{\frac{N}{m}}\, \widehat d_N(j \, m)$ diverges to $\infty$ when $d\to 1.25$, see Surgailis {\it et al.}, 2008). \\
Now define
\begin{equation}\label{Sigma}
\widehat \Sigma_N(m):=(\Lambda'_0(\widehat d_N(m))^{-2}\, \Gamma_p(\widehat d_N(m)).
\end{equation}
The function $d\in (-0.5,1.5) \mapsto \sigma(d)/\Lambda'(d)$ is ${\cal C}^\infty$ and therefore, under assumptions of Theorem \ref{cltnada},
$$\widehat \Sigma_N(m) \limiteprobaN (\Lambda'_0(d))^{-2}\, \Gamma_p(d).$$Thus, a pseudo-generalized least square estimation (LSE) of $d$ ican be defined by
$$
\widetilde d_N(m):=\big (J_p^{\intercal} \big (\widehat \Sigma_N(m)\big )^{-1} J_p \big )^{-1}\, J_p^{\intercal} \,  \big ( \widehat \Sigma_N(m) \big )^{-1} \big (\widehat d_N(m_i) \big ) _{1\leq i \leq p}
$$
with $J_p:=(1)_{1\leq j \leq p}$ and denoting $J_p^{\intercal}$ its transpose.
From Gauss-Markov Theorem, the asymptotic variance of $\widetilde d_N(m)$ is smaller than the one of $\widehat d_N(jm)$, $j=1,\ldots,p$. Hence, we obtain under the assumptions of Theorem \ref{cltnada}:
\begin{eqnarray}\label{TLCdtilde}
\sqrt{\frac{N}{m}}\big (\widetilde d_N(m)-d\big ) \limiteloiN {\cal N}\Big(0 \,, \, \Lambda'_0(d)^{-2}\,\big (J_p^{\intercal} \, \Gamma^{-1}_p(d)J_p\big )^{-1}\Big ).
\end{eqnarray}
%\noindent Note that this test is also a test of long memory when $d>0$ and it is very simple to be implemented.
\section{The adaptive version of the estimator}\label{Adapt}
Theorem \ref{cltnada} and CLT \eqref{TLCdtilde} require the knowledge of $\beta$ to be applied. But in practice $\beta$ is unknown. The procedure defined in Bardet {\em et al.} (2008) or Bardet and Dola (2012) can be used for obtaining a data-driven selection of an optimal sequence $(\widetilde m_N)$ derived from an estimation of $\beta$. Since the case $d \in (-0.5,0.5)$ was studied in Bardet and Dola (2012) we consider here $d \in [0.5,1.25)$ and for $\alpha \in (0,1)$, define
\begin{equation}\label{QNdef}
Q_N(\alpha,d):=\big (\widehat d_N(j\, N^\alpha)-\widetilde d_N(N^{\alpha}) \big )^{\intercal}_{1\leq j \leq p} \big (\widehat \Sigma_N(N^\alpha)\big )^{-1}\big (\widehat d_N(j\, N^\alpha)-\widetilde d_N(N^{\alpha}) \big )_{1\leq j \leq p},
\end{equation}
which corresponds to the sum of the pseudo-generalized squared distance between the points $(\widehat d_N(j\, N^\alpha))_j$ and PGLS estimate of $d$.
Note that by the previous convention, $\widehat d_N(j\, N^\alpha)=\widehat d_N(j\, [N^\alpha])$ and $\widetilde d_N(N^\alpha)=\widetilde d_N([N^\alpha])$. Then $\widehat Q_N(\alpha)$ can be minimized on a discretization of $(0,1)$ and define:
\begin{eqnarray*}
\widehat \alpha_N :=\mbox{Argmin}_{\alpha \in {\cal A}_N}
\widehat Q_N(\alpha)\quad \mbox{with}\quad {\cal A}_N=\Big \{\frac {2}{\log
N}\,,\,\frac { 3}{\log N}\,, \ldots,\frac {\log [N/p]}{\log N}
\Big \}.
\end{eqnarray*}
\begin{rem}\label{defAn}
The choice of the set of discretization ${\cal A}_N$ is implied by our proof of convergence of $\widehat \alpha_N$ to $\alpha^*$. If the interval $(0,1)$ is stepped in $N^c$ points, with $c>0$, the used proof cannot
attest this convergence. However $\log N$ may be replaced in the previous expression of ${\cal A}_N$ by any negligible function
of $N$ compared to functions $N^c$ with $c>0$ (for instance, $(\log N)^a$ or $a\log N$ with $a>0$ can be used).
\end{rem}
\noindent From the central limit theorem ({\ref{cltd}}) one deduces the following limit theorem:
\begin{prop}\label{hatalpha}
Assume that Assumption $IG(d,\beta)$ holds with $0.5\leq d<1.25$ and $0<\beta\leq 2$. Then,
$$
\widehat \alpha_N
\limiteprobaN \alpha^*=\frac 1 {(1+2\beta)}.%\quad \mbox{when}\quad \beta\leq 1+2d.
$$
\end{prop}
\noindent

Finally define
$$
\widetilde
m_N:=N^{\widetilde \alpha_N}\quad \mbox{with}\quad \widetilde \alpha_N:=\widehat \alpha_N+ \frac {6\, \widehat \alpha_N} {(p-2)(1-\widehat
{\alpha}_N )} \cdot \frac {\log \log N}{\log N}.
$$
and the estimator
\begin{equation}\label{dtilde}
\widetilde d_N^{IR} :=\widetilde
d_N(\widetilde m_N)=\widetilde d_N(N^{\widetilde \alpha_N}).
\end{equation}
(the definition and use of $\widetilde \alpha_N$ instead of $\widehat \alpha_N$ are explained just before Theorem 2 in Bardet and Dola, 2012).  The following theorem provides
the asymptotic behavior of the estimator $\widetilde d_N^{IR}$:
\begin{thm}\label{tildeD}
Under assumptions of Proposition \ref{hatalpha},
\begin{eqnarray}\label{CLTD2}
&&\sqrt{\frac{N}{N^{\widetilde \alpha_N }}} \big(\widetilde d_N^{IR}  - d \big)
\limiteloiN    {\cal N}\Big (0\, ; \,\Lambda'_0(d)^{-2}\,\big (J_p^{\intercal} \, \Gamma^{-1}_p(d)J_p\big )^{-1}\Big ).
\end{eqnarray}
Moreover, $\displaystyle ~~\forall
\rho>\frac {2(1+3\beta)}{(p-2)\beta},~\mbox{}~~ \frac {N^{\frac
{\beta}{1+2\beta}} }{(\log N)^\rho} \cdot \big|\widetilde d_N^{IR} - d \big|
\limiteprobaN  0.$
\end{thm}
The convergence rate of $\widetilde d_N^{IR}$ is the same (up to a multiplicative logarithm factor) than the one of minimax estimator of $d$ in this semiparametric frame (see Giraitis {\it et al.}, 1997). The supplementary advantage of $\widetilde d_N^{IR}$ with respected to other adaptive estimators of $d$ (see for instance Moulines and Soulier, 2003, for an overview about frequency domain estimators of $d$) is the central limit theorem (\ref{CLTD2}) satisfied by $\widetilde d_N^{IR}$. Moreover $\widetilde d_N^{IR}$ can be used for $d\in (-0.5,1.25)$, {\it i.e.} as well for stationary and non-stationary processes, without modifications in its definition. Both this advantages allow to define stationarity and nonstationarity tests based on $\widetilde d_N^{IR}$.

\section{Stationarity and nonstationarity tests}\label{test}
Assume that $(X_1,\ldots,X_N)$ is an observed trajectory of a process $X=(X_k)_{k\in \Z}$. We define here new stationarity and nonstationarity tests for $X$ based on $\widetilde d_N^{IR}$.
\subsection{A stationarity test}
There exist many stationarity and nonstationarity test. The most famous stationarity tests are certainly the following unit root tests:
\begin{itemize}
\item The KPSS (Kwiotowski, Phillips, Schmidt, Shin) test  (see for instance Hamilton, 1994, p. 514);
\item The LMC (Leybourne, McCabe) test which is a generalization of the KPSS test (see for instance Leybourne and  McCabe, 1994 and 1999).
\end{itemize}
We can also cite the V/S test (see its presentation in Giraitis {\it et al.}, 2001) which was first defined for testing the presence of long-memory versus short-memory. As it was already notified in Giraitis {\it et al.} (2003-2006), the V/S test is also more powerful than the KPSS test for testing the stationarity. \\
~\\
More precisely, we consider here the following problem of test:
\begin{itemize}
\item \underline{Hypothesis $H_0$} (stationarity):  $(X_t)_{t\in \Z}$ is a process satisfying Assumption $IG(d,\beta)$ with $d\in [-a_0,a'_0]$ where $0\leq a_0,\, a'_0<1/2$  and $\beta \in [b_0,2]$ where $0<b_0\leq 2$.
\item \underline{Hypothesis $H_1$} (nonstationarity):  $(X_t)_{t\in \Z}$ is a process satisfying Assumption $IG(d,\beta)$ with $d\in [0.5,a_1] $ where $0\leq a_1<1.25$  and $\beta \in [b_1,2]$ where $0<b_1\leq 2$.
\end{itemize}
We use a test based on $\widetilde d_N^{IR}$ for deciding between these hypothesis. Hence from the previous CLT \ref{CLTD2} and with a type I error $\alpha$, define
\begin{equation}\label{SNtild}
\widetilde S_N:= \1_{ \widetilde{d}_{N}^{IR}>0.5+\sigma_p(0.5)\,   q_{1-\alpha} \,   N^{(\widetilde{\alpha}_N-1)/2}},
\end{equation}
where $\sigma_p(0.5)=\Big ( \Lambda'_0(0.5)^{-2}\,\big (J_p^{\intercal} \, \Gamma^{-1}_p(0.5)J_p\big )^{-1}\Big )^{1/2}$(see \eqref{CLTD2}) and $q_{1-\alpha}$ is the $(1-\alpha)$ quantile of a standard Gaussian random variable ${\cal N}(0,1)$. \\
~\\
Then we define the following rules of decision:
\begin{itemize}
\item $H_0$ (stationarity) is accepted when $\widetilde S_N=0$ and rejected when $\widetilde S_N=1$.
\end{itemize}
\begin{rem}
In fact, the previous stationarity test $\widetilde S_N$ defined in \eqref{SNtild} can also be seen as a semi-parametric test $d<d_0$ versus $d \geq d_0$ with $d_0=0.5$. It is obviously possible to extend it to any value $d_0 \in (-0.5,1.25)$ by defining 
$
\widetilde S^{(d_0)}_N:= \1_{ \widetilde{d}_{N}^{IR}>d_0+\sigma_p(d_0)\,   q_{1-\alpha} \,   N^{(\widetilde{\alpha}_N-1)/2}}.
$
\end{rem}
From previous results, it is clear that:
\begin{popy} \label{StatIG}
Under Hypothesis $H_0$, the asymptotic type I error of the test $\widetilde S_N$ is $\alpha$ and under Hypothesis $H_1$, the test power tends to $1$.
\end{popy}

Moreover, this test can be used as a unit root test. Indeed, define the following typical problem of unit root test. Let $X_t=at+b+\varepsilon_t$, with $(a,b)\in \R^2$, and $\varepsilon_t$ an ARIMA$(p,d,q)$ with $d=0$ or $d=1$. Then, a (simplified)  problem of a unit root test is to decide between:
\begin{itemize}
\item $H^{UR}_0$: $d=0$ and $(\varepsilon_t)$ is a stationary ARMA$(p,q)$ process.
\item $H^{UR}_1$: $d=1$ and $(\varepsilon_t-\varepsilon_{t-1})_t$ is a stationary ARMA$(p,q)$ process.
\end{itemize}
Then,
\begin{popy}\label{StatUR}
Under assumption $H^{UR}_0$, the type I error of this unit root test problem using $\widetilde S_N$ decreases to $0$ when $N \to \infty$ and the test power tends to $1$.
\end{popy}

\subsection{A new nonstationarity test}
Famous unit root tests are more often nonstationarity test. For instance, between the most famous tests,
\begin{itemize}
\item The Augmented Dickey-Fuller test (see Hamilton, 1994, p. 516-528 for details);
\item The Philipps and Perron test (a generalization of the ADF test with more lags, see for instance Elder, 2001, p. 137-146).
\end{itemize}
Using the statistic $\widetilde{d}_{N}^{IR}$ we propose a new nonstationarity test $\widetilde T_N$ for deciding between:
\begin{itemize}
\item \underline{Hypothesis $H'_0$} (nonstationarity):  $(X_t)_{t\in \Z}$ is a process satisfying Assumption $IG(d,\beta)$ with $d\in [0.5,a'_0]$ where $0.5\leq a'_0<1.25$  and $\beta \in [b'_0,2]$ where $0<b'_0\leq 2$. 
\item \underline{Hypothesis $H'_1$} (stationarity):  $(X_t)_{t\in \Z}$ is a process satisfying Assumption $IG(d,\beta)$ with $d\in [-a'_1,b'_1]$ where $0\leq a'_1,\, b'_1<1/2$  and $\beta \in [c'_1,2]$ where $0<c'_1\leq 2$.
\end{itemize}
Then, the rule of the test is the following: Hypothesis $H'_0$ is accepted when $\widetilde T_N=1$ and rejected when $\widetilde T_N=0$ where
\begin{equation}\label{TNtild}
\widetilde T_N:= \1_{ \widetilde{d}_{N}^{IR}<0.5-\sigma_p(0.5)\,   q_{1-\alpha} \,   N^{(\widetilde{\alpha}_N-1)/2}}.
\end{equation}
Then as previously
\begin{popy} \label{NStatIG}
Under Hypothesis $H_0'$, the asymptotic type I error of the test $\widetilde T_N$ is $\alpha$ and under Hypothesis $H_1'$ the test power tends to $1$.
\end{popy}
As previously, this test can also  be used as a unit root test where $X_t=at+b+\varepsilon_t$, with $(a,b)\in \R^2$, and $\varepsilon_t$ an ARIMA$(p,d,q)$ with $d=0$ or $d=1$. We consider here a ``second'' simplified problem of unit root test which is to decide between:
\begin{itemize}
\item $H^{UR'}_0$: $d=1$ and $(\varepsilon_t-\varepsilon_{t-1})_t$ is a stationary ARMA$(p,q)$ process.
\item $H^{UR'}_1$: $d=0$ and $(\varepsilon_t)_t$ is a stationary ARMA$(p,q)$ process..
\end{itemize}
Then,
\begin{popy}\label{StatUR2}
Under assumption $H^{UR'}_0$, the type I error of the unit root test problem using $\widetilde T_N$ decreases to $0$ when $N \to \infty$ and the test power tends to $1$.
\end{popy}

\section{Results of simulations and application to Econometric and Financial data}\label{simu}
\subsection{Numerical procedure for computing the estimator and tests}
First of all, softwares used in this Section are available on {\tt
http://samm.univ-paris1.fr/-Jean-Marc-Bardet} with a free access on (in Matlab language) as well as classical estimators or tests.\\
~\\
The concrete procedure for applying our MIR-test of stationarity is the following:
\begin{enumerate}
 \item using additional simulations (realized on ARMA, ARFIMA, FGN processes and not presented here for avoiding too much expansions), we have observed that the value of the parameter $p$ is not really important with respect to the accuracy of the test (less than $10\%$ on the value of $\widetilde d^{IR}_N$). However, for optimizing our procedure we chose $p$ as a stepwise function of $n$:
$$
p=5\times \1_{\{n<120\}}+ 10\times \1_{\{120\leq n<800\}}+ 15\times \1_{\{800\leq n<10000\}}+20\times \1_{\{n\geq 10000\}}
$$
and
$\sigma_{5}(0.5)\simeq 0.9082; ~\sigma_{10}(0.5) \simeq 0.8289; ~ \sigma_{15}(0.5) \simeq 0.8016~\mbox{and}~  \sigma_{20}(0.5)\simeq 0.7861$.
\item then using the computation of $\widetilde m_N$ presented in Section \ref{Adapt}, the adaptive estimator $\widetilde d^{IR}_N$ (defined in \eqref{dtilde}) and the test statistics  $\widetilde{S}_{N}$ (defined in \eqref{SNtild}) and $\widetilde{T}_{N}$ (defined in \eqref{TNtild}) are computed.
\end{enumerate}
\subsection{Monte-Carlo experiments on several time series}
In the sequel the results are obtained from $300$ generated independent samples of
each process defined below. The concrete
procedures of generation of these processes are obtained from the
circulant matrix method, as detailed in Doukhan {\it et al.} (2003).
The simulations are realized for different values of $d$ and
$N$ and processes which satisfy Assumption $IG(d,\beta)$:
\begin{enumerate}
\item the usual ARIMA$(p,d,q)$ processes with respectively $d=0$ or $d=1$ and an innovation process which is a Gaussian white noise.
Such processes satisfy
{Assumption} $IG(0,2)$ or $IG(1,2)$ holds (respectively);
\item the ARFIMA$(p,d,q)$ processes with parameter $d$ such that
$d \in (-0.5,1.25)$ and an innovation process which is a Gaussian white noise.
Such ARFIMA$(p,d,q)$  processes satisfy Assumption $IG(d,2)$ (note that ARIMA processes are particular cases of ARFIMA processes).

\item the Gaussian stationary processes $X^{(d,\beta)}$, such as its spectral density is
\begin{eqnarray}
f_3(\lambda)=\frac 1 {|\lambda|^{2d}}(1+c_1\, |\lambda|^{\beta})~~~\mbox{for
$\lambda \in [-\pi,0)\cup (0,\pi]$},
\end{eqnarray}
with $d \in (-0.5,1.5)$, $c_1>0$ and $\beta\in (0,\infty)$. Therefore the spectral density $f_{3}$ implies that
Assumption $IG(d,\beta)$ holds. In the sequel we will use $c_1=5$ and $\beta=0.5$, implying that the second order term of the spectral density is less negligible than in case of FARIMA processes.
\end{enumerate}

\subsubsection*{Comparison of $\widetilde d_N^{IR}$  with other semiparametric estimators of $d$}
Here we first compare the performance of the adaptive MIR-estimator $\widetilde d_N^{IR}$ with other famous semiparametric estimators of $d$:
\begin{itemize}

\item $\widetilde{d}_N^{MS}$ is the adaptive global log-periodogram estimator introduced
by Moulines and Soulier (2003), also called FEXP estimator,
with bias-variance balance parameter $\kappa=2$. Such an estimator was shown to be consistent for $d\in ]-0.5,1.25]$.
\item $\widehat d_N^{ADG}$ is the extended local Whittle estimator defined by
Abadir, Distaso and Giraitis (2007) which is consistent for $d>-3/2$. It is a generalization of the local Whittle estimator introduced by Robinson (1995b), consistent for $d<0.75$, following a first extension proposed by Phillips (1999)  and Shimotsu and Phillips (2005). This estimator avoids the tapering used for instance in Velasco (1999b) or Hurvich and Chen (2000). The trimming parameter is chosen as $m=N^{0.65}$ (this is not an adaptive estimator) following the numerical recommendations of Abadir {\it et al.} (2007).
\item $\widetilde{d}_N^{WAV}$ is an adaptive wavelet based estimator
introduced in Bardet {\it et al.} (2013) using a Lemarie-Meyer type wavelet (another similar choice could be the adaptive wavelet estimator introduced
in Veitch {\it et al.}, 2003, using a Daubechie's wavelet, but its robustness property are slightly less interesting). The asymptotic normality of such estimator is established for $d\in \R$ (when the number of vanishing moments of the wavelet function is large enough).
\end{itemize}
Note that  only  $\widehat d_N^{ADG}$ is the not adaptive among the $4$ estimators. 
Table \ref{Table1}, \ref{Table2}, \ref{Table3} and \ref{Table4} respectively provide the results of simulations for ARIMA$(1,d,0)$, ARFIMA$(0,d,0)$, ARFIMA$(1,d,1)$ and  $X^{(d,\beta)}$ processes for several values of $d$ and $N$.  \\
~\\
\begin{table}[t]
{\footnotesize
\begin{center}
\begin{tabular}{|c|c|c|c||c|c|c|}
\hline\hline
  $N=500$  & $d=0$ & $d=0$ & $d=0$  &$d=1$ & $d=1$  & $d=1$ \\
ARIMA$(1,d,0)$ &$\phi$=-0.5 &$\phi$=-0.7 &$\phi$=-0.9 &$\phi$=-0.1 &$\phi$=-0.3  &$\phi$=-0.5  \\
%&              &$\rho=0$  &$\rho=0$  &$\rho=0$  &$\rho=0$  &$\rho=1$  &$\rho=1$  &$\rho=1$  &$\rho=1$ \\
\hline \hline
$\sqrt{MSE}$ $\widetilde{d}_N^{IR}$    &0.163   &0.265    &0.640    &0.093    &0.102   &\textbf{0.109}
\\
\hline
$\sqrt{MSE}$ $\widetilde{d}_N^{MS}$  &0.138    &\textbf{0.148}    &\textbf{0.412}   &0.172    &0.163    &0.170
\\
\hline
$\sqrt{MSE}$ $\widehat d_N^{ADG}$  &\textbf{0.125}    &0.269   &0.679    &0.074   &\textbf{0.078}    &0.120
\\
\hline
$\sqrt{MSE}$ $\widetilde{d}_N^{WAV}$  &0.246   &0.411    &0.758   &\textbf{0.067}  &0.099    &0.133
\\
\hline

\end{tabular}\\
~\\\vspace{3mm}
\begin{tabular}{|c|c|c|c||c|c|c|}
\hline\hline
  $N=5000$  & $d=0$ & $d=0$ & $d=0$ &$d=1$ & $d=1$ & $d=1$  \\
ARFIMA(1,d,0) &$\phi$=-0.5 &$\phi$=-0.7 &$\phi$=-0.9 &$\phi$=-0.1 &$\phi$=-0.3  &$\phi$=-0.5  \\
%&              &$\rho=0$  &$\rho=0$  &$\rho=0$  &$\rho=0$  &$\rho=1$  &$\rho=1$  &$\rho=1$  &$\rho=1$ \\
\hline \hline
$\sqrt{MSE}$ $\widetilde{d}_N^{IR}$   &0.077    &0.106    &0.293    &\textbf{0.027}   &0.048    &0.062
\\
\hline
$\sqrt{MSE}$ $\widetilde{d}_N^{MS}$   &0.045    &\textbf{0.050}    &0.230   &0.046    &0.046    &0.040
\\
\hline
$\sqrt{MSE}$ $\widehat d_N^{ADG}$   &\textbf{0.043}    &0.085    &0.379    &0.031    &\textbf{0.032}    &\textbf{0.036}
\\
\hline
$\sqrt{MSE}$ $\widetilde{d}_N^{WAV}$    &0.080   &0.103    &\textbf{0.210}   &0.037    &0.044    &0.054
\\
\hline
\end{tabular}
\end{center}
}
\caption{{\small\label{Table1} : Comparison between $\widetilde{d}^{IR}_N$  and other famous semiparametric estimators of $d$ ($\widetilde{d}_N^{MS}$, $\widehat d_N^{ADG}$  and $\widetilde{d}_N^{WAV}$) applied to ARIMA$(1,d,0)$ process ($(1-B)^d(1+\phi\, B)\, X=\varepsilon$), with several $\phi$ and $N$ values}
}

\end{table}

\begin{table}[t]
{\footnotesize
\begin{center}
\begin{tabular}{|c|c|c|c|c||c|c|c|c|}
\hline\hline
$N=500$ &$d=-0.2$ &$d=0$ &$d=0.2$ &$d=0.4$  &$d=.6$ &$d=0.8$ &$d=1$ &$d=1.2$\\
ARFIMA(0,d,0) &&&&&&&&\\
\hline \hline
$\sqrt{MSE}$ $\widetilde{d}_N^{IR}$  &0.088    &0.092    &0.097    &0.096     &0.101    &0.101 &0.099    &0.105
\\
\hline
$\sqrt{MSE}$ $\widetilde{d}_N^{MS}$   &0.144    &0.134    &0.146    &0.152      &0.168    &0.175 &0.165    &0.157
\\
\hline
$\sqrt{MSE}$ $\widehat d_N^{ADG}$     &0.075    &\textbf{0.078}    &\textbf{0.080} &\textbf{0.084}      &\textbf{0.083}    &\textbf{0.079} &0.077    &0.081
\\
\hline
$\sqrt{MSE}$ $\widetilde{d}_N^{WAV}$   &\textbf{0.071}    &0.079    &0.087   &0.088   &0.087    &0.085 &\textbf{0.069}    &\textbf{0.076}
\\
\hline

\end{tabular}\\
~\\
\vspace{3mm}\begin{tabular}{|c|c|c|c|c||c|c|c|c|}
\hline\hline
$N=5000$ &$d=-0.2$ &$d=0$ &$d=0.2$ &$d=0.4$ &$d=.6$ &$d=0.8$ &$d=1$ &$d=1.2$\\
ARFIMA(0,d,0) &&&&&&&&\\
\hline \hline
$\sqrt{MSE}$ $\widetilde{d}_N^{IR}$   &0.037  &\textbf{0.025}   &\textbf{0.031}  &0.031     &0.035    &0.035 &0.038    &0.049
\\
\hline
$\sqrt{MSE}$ $\widetilde{d}_N^{MS}$    &0.043   &0.042    &0.043    &0.042     &0.055    &0.054 &0.046    &0.147
\\
\hline
$\sqrt{MSE}$ $\widehat d_N^{ADG}$    &0.034    &0.033    &0.032    &0.036     &0.033  &\textbf{0.032} &\textbf{0.033}    &\textbf{0.032}
\\
\hline
$\sqrt{MSE}$ $\widetilde{d}_N^{WAV}$   &\textbf{0.033}    &0.032   &\textbf{0.031}    &\textbf{0.023}     &\textbf{0.023}    &0.038 &0.039    &0.041
\\
\hline
\end{tabular}
\end{center}

\caption{{\small \label{Table2}: Comparison between $\widetilde{d}^{IR}_N$  and other famous semiparametric estimators of $d$ ($\widetilde{d}_N^{MS}$, $\widehat d_N^{ADG}$  and $\widetilde{d}_N^{WAV}$) applied to ARFIMA$(0,d,0)$ process, with several $d$ and $N$ values}
}
}
\end{table}

\begin{table}[t]
{\footnotesize
\begin{center}

\begin{tabular}{|c|c|c|c|c||c|c|c|c|}
\hline\hline
$N=500$ &&&&&&&&\\
ARFIMA(1,d,1) &$d=-0.2$ &$d=0$ &$d=0.2$ &$d=0.4$ &$d=.6$ &$d=0.8$ &$d=1$ &$d=1.2$\\
$\phi=-0.3$ ; $\theta=0.7$ &&&&&&&&\\
\hline \hline
$\sqrt{MSE}$ $\widetilde{d}_N^{IR}$  &0.152    &0.132   &0.125   &0.125       &0.118   &0.117 &0.111    &0.112
\\
\hline
$\sqrt{MSE}$ $\widetilde{d}_N^{MS}$   &0.138    &0.137    &0.144    &0.155    &0.161    &0.179 &0.172    &0.170
\\
\hline
$\sqrt{MSE}$ $\widehat d_N^{ADG}$ &\textbf{0.092}    &\textbf{0.088}    &\textbf{0.090}    &\textbf{0.097}      &\textbf{0.096}    &\textbf{0.087} &\textbf{0.087}    &\textbf{0.087}
\\
\hline
$\sqrt{MSE}$ $\widetilde{d}_N^{WAV}$    &0.173    &0.154   &0.152    &0.148   &0.139    &0.132 &0.105  &0.098
\\
\hline
\end{tabular} \\
~\\
\vspace{3mm}
\begin{tabular}{|c|c|c|c|c||c|c|c|c|}
\hline\hline
$N=5000$ &&&&&&&&\\
ARFIMA(1,d,1)  &$d=-0.2$ &$d=0$ &$d=0.2$ &$d=0.4$ &$d=.6$ &$d=0.8$ &$d=1$ &$d=1.2$\\
$\phi=-0.3$ ; $\theta=0.7$ &&&&&&&&\\
\hline \hline
$\sqrt{MSE}$ $\widetilde{d}_N^{IR}$    &0.070    &0.062    &0.053    &0.052       &0.052    &0.054 &0.059   &0.58\\
\hline
$\sqrt{MSE}$ $\widetilde{d}_N^{MS}$    &\textbf{0.038}   &0.042    &0.041    &0.050       &0.052    &0.054 &0.045    &0.150\\
\hline
$\sqrt{MSE}$ $\widehat d_N^{ADG}$  &0.039   &\textbf{0.035}  &\textbf{0.033} &\textbf{0.037}    &\textbf{0.038}  &\textbf{0.037} &\textbf{0.035}    &\textbf{0.033}\\
\hline
$\sqrt{MSE}$ $\widetilde{d}_N^{WAV}$    &0.049   &0.057    &0.056    &0.053      &0.051    &0.050 &0.048    &0.050\\
\hline
\end{tabular}
\end{center}
}
\caption{{\small\label{Table3}:  Comparison between $\widetilde{d}^{IR}_N$  and other famous semiparametric estimators of $d$ ($\widetilde{d}_N^{MS}$, $\widehat d_N^{ADG}$  and $\widetilde{d}_N^{WAV}$) applied to ARFIMA$(1,d,1)$ process (with $\phi=-0.3$ and $\theta=0.7$), with several $d$ and $N$ values.}
}
\end{table}

\begin{table}[t]
{\footnotesize
\begin{center}
\begin{tabular}{|c|c|c|c|c||c|c|c|c|c|}
\hline\hline
$N=500$  &$d=-0.2$ &$d=0$ &$d=0.2$ &$d=0.4$ &$d=.6$ &$d=0.8$ &$d=1$ &$d=1.2$\\
$X^{(d,\beta)}$ &&&&&&&&\\
\hline \hline
$\sqrt{MSE}$ $\widetilde{d}_N^{IR}$  &\textbf{0.140}   &\textbf{0.170}  &0.201   &0.211     &0.209 &0.205 &0.210    & 0.202\\
\hline
$\sqrt{MSE}$ $\widetilde{d}_N^{MS}$   &0.187    &0.188    &0.204   &0.200   &0.192    &0.187 &0.200   &0.192\\
\hline
$\sqrt{MSE}$ $\widehat d_N^{ADG}$   &0.177   &0.182    &\textbf{0.190}    &\textbf{0.184}     &\textbf{0.174}    &\textbf{0.179} &0.196    &0.189\\
\hline
$\sqrt{MSE}$ $\widetilde{d}_N^{WAV}$   &0.224    &0.225    &0.230    &0.220      &0.213    &0.199&\textbf{0.185}   &\textbf{0.175}\\
\hline
\end{tabular}\\
~\\
\vspace{3mm}
\begin{tabular}{|c|c|c|c|c||c|c|c|c|}
\hline\hline
$N=5000$&$d=-0.2$ &$d=0$ &$d=0.2$ &$d=0.4$ &$d=.6$ &$d=0.8$ &$d=1$ &$d=1.2$\\
$X^{(d,\beta)}$ &&&&&&&&\\
\hline \hline
$\sqrt{MSE}$ $\widetilde{d}_N^{IR}$  &\textbf{0.110}  &0.139   &0.150    &0.151     &0.152    &0.153 &0.152    &\textbf{0.142}
\\
\hline
$\sqrt{MSE}$ $\widetilde{d}_N^{MS}$ &0.120    &\textbf{0.123}   &\textbf{0.132}    &\textbf{0.131}     &\textbf{0.132}    &\textbf{0.127} &\textbf{0.134}    &0.155
\\
\hline
$\sqrt{MSE}$ $\widehat d_N^{ADG}$   &0.139    &0.138    &0.141    &0.134      &0.134    &0.140 &0.140   &0.145
\\
\hline
$\sqrt{MSE}$ $\widetilde{d}_N^{WAV}$   &0.170    &0.173    &0.167    &0.165      &0.167    &0.166 &0.164    &0.150
\\
\hline
\end{tabular}
\end{center}

\caption{\small {\label{Table4}: Comparison between $\widetilde{d}^{IR}_N$  and other famous semiparametric estimators of $d$ ($\widetilde{d}_N^{MS}$, $\widehat d_N^{ADG}$  and $\widetilde{d}_N^{WAV}$) applied to $X^{(d,\beta)}$ process with several $d$ and $N$ values.}
}
}
\end{table}
~\\
\underline{{\bf Conclusions of simulations:}} In almost $50 \%$ of cases (especially for $N=500$), the estimator  $\widehat d_N^{ADG}$  provides the smallest $\sqrt{MSE}$ among the $4$ semiparametric estimators even if this estimator is not an adaptive estimator (the bandwidth $m$ is fixed to be $N^{0.65}$, which should theretically be a problem when $2\beta(2\beta+1)^{-1}<0.65$, {\it i.e.} $\beta<13/14$). However, even for the process $X^{(d,\beta)}$ with $\beta=0.5$, the estimator $\widehat d_N^{ADG}$ provides not so bad results (when $N=5000$, note that $\widehat d_N^{ADG}$ never provides the best results contrarly to what happens with the $3$ other processes sayisfying $\beta=2$). Some additional simulations, not reported here, realized with $N=10^5$ always for $X^{(d,\beta)}$ with $\beta=0.5$, show that the $\sqrt{MSE}$  of $\widehat d_N^{ADG}$ becomes the worst (the largest) among the $\sqrt{MSE}$ of the $4$ other estimators. The estimator $\widetilde d_N^{IR}$ provide convincing results, almost the same performances than the other adaptive estimators $\widetilde{d}_N^{MS}$  and $\widetilde{d}_N^{WAV}$. 

\subsubsection*{Comparison of MIR tests $\widetilde S_N$ and $\widetilde T_N$  with other famous stationarity or nonstationarity tests}
Monte-Carlo experiments were done for evaluating the performances of new tests $\widetilde S_N$ and $\widetilde T_N$ and for comparing them to most famous stationarity tests (LMC and V/S, V/S replacing KPSS) or nonstationarity (ADF and PP) tests (see more details on these tests in the previous section). \\
We also defined a stationarity and nonstationarity test based on  the extended local Whittle estimator $\widehat d_N^{ADG}$ following the results obtained in Abadir {\it et al.} (2007) (a very simple central limit theorem was stated in Corollary 2.1). Then, for instance, the stationarity test $\widehat S_{ADG}$ is defined by 
$$
\widehat S^{ADG}:= \1_{ \widehat d_N^{ADG}>0.5+\frac 1 2 \,   q_{1-\alpha} \,  \frac {1}{\sqrt{ m}}},
$$
with $m=N^{0.65 }$ (and the nonstationarity test $\widehat T_{ADG}$ is built following the same trick).
\begin{itemize}
 %\item $k=\Big [\frac 3 {13}\, \sqrt n \Big ]$ for KPSS test;
\item $k=0$ for LMC test;
\item $k=\sqrt n$ for V/S test;
\item $k=\Big [(n-1)^{1/3}\Big ]$ for ADF test;
\item  $k=\Big [4\, \big (\frac n {100} \big )^{1/4}\Big ]$ for PP test;
\end{itemize}
The results of these simulations  with a type I error classically chosen to $0.05$ are provided in Tables \ref{Table5}, \ref{Table6}, \ref{Table7} and \ref{Table8}.
\\

\begin{table}
{\footnotesize
\begin{center}
\begin{tabular}{|c|c|c|c||c|c|c|}
\hline\hline
 $N=500$  & $d=0$ & $d=0$  & $d=0$ &$d=1$ & $d=1$ & $d=1$  \\
ARIMA$(1,d,0)$ &$\phi$=-0.5 &$\phi$=-0.7 &$\phi$=-0.9 &$\phi$=-0.1 &$\phi$=-0.3  &$\phi$=-0.5 \\
%&              &$\rho=0$  &$\rho=0$  &$\rho=0$  &$\rho=0$  &$\rho=1$  &$\rho=1$  &$\rho=1$  &$\rho=1$ \\
\hline \hline
$\widetilde S_N$: Accepted $H_{0}$   &1   &1    &0.37    &0    &0    &0 \\

$\widehat S_{ADG}$: Accepted $H_{0}$   &1    &1    &0.25    &0    &0    &0
\\
LMC: Accepted $H_{0}$     &0.97    &1    &0.84    &0.02    &0         &0
\\
$V/S$ : Accepted $H_{0}$    &0.96    &0.93         &0.84    &0.09         &0.08         &0.12\\
\hline
\hline
$\widetilde T_N$: Rejected $H'_{0}$    &0.99        &0.77         &0.08         &0         &0         &0
\\
$\widehat T_{ADG}$: Rejected $H'_{0}$   &1         &0.94         &0         &0         &0         &0
\\
ADF: Rejected $H'_{0}$    &1    &1    &1   &0.06    &0.04    &0.04 \\

PP : Rejected $H'_{0}$    &1    &1    &1    &0.06    &0.03    &0.02\\
\hline
\hline
\end{tabular}\\
~\\
\begin{tabular}{|c|c|c|c||c|c|c|}
\hline\hline
 $N=5000$  & $d=0$ & $d=0$  & $d=0$ &$d=1$ & $d=1$ & $d=1$   \\
ARIMA$(1,d,0)$ &$\phi$=-0.5 &$\phi$=-0.7 &$\phi$=-0.9 &$\phi$=-0.1 &$\phi$=-0.3  &$\phi$=-0.5 \\
%&              &$\rho=0$  &$\rho=0$  &$\rho=0$  &$\rho=0$  &$\rho=1$  &$\rho=1$  &$\rho=1$  &$\rho=1$ \\
\hline \hline
$\widetilde S_N$: Accepted $H_{0}$   &1   &1    &0.91    &0    &0    &0 \\

$\widehat S_{ADG}$: Accepted $H_{0}$   &1    &1    &1   &0    &0    &0
\\
LMC: Accepted $H_{0}$     &0.95   &1    &1    &0   &0         &0
\\
$V/S$ : Accepted $H_{0}$    &0.93   &0.97         &0.90    &0        &0         &0\\
\hline
\hline
$\widetilde T_N$: Rejected $H'_{0}$   &1        &1         &0.87         &0         &0         &0
\\
$\widehat T_{ADG}$: Rejected $H'_{0}$     &1         &1         &0.95         &0         &0         &0
\\
ADF: Rejected $H'_{0}$    &1    &1    &1    &0.09    &0.01    &0.04 \\

PP : Rejected $H'_{0}$    &1    &1    &1    &0.07    &0.01    &0.04
\\
\hline
\hline
\end{tabular}
\end{center}
}
\caption{{\small \label{Table5} Comparisons of stationarity and nonstationarity tests from $300$ independent replications of ARIMA$(1,d,0)$ processes ($X_t+\phi X_{t-1}=\varepsilon_t$) for  several values of $\phi$ and $N$. The accuracy of tests is measured by the frequencies of  trajectories ``accepted as stationary'' (accepted $H_0$ or rejected $H'_0$) among the $300$ replications which should be close to $1$ for $d\in(-0.5,0.5)$ and close to $0$ for $d\in[0.5,1.2]$}
}
\end{table}

\begin{table}
{\footnotesize
\begin{center}
\begin{tabular}{|c|c|c|c|c||c|c|c|c|}
\hline\hline
$N=500$ &&&&&&&&\\
ARFIMA$(0,d,0)$ &$d=-0.2$ &$d=0$ &$d=0.2$ &$d=0.4$ &$d=0.6$ &$d=0.8$ &$d=1$ &$d=1.2$ \\
\hline \hline
$\widetilde S_N$: Accepted $H_{0}$    &1    &1   &1    &1    &0.72    &0.09
&0.01         &0 \\
$\widehat S_{ADG}$: Accepted $H_{0}$   &1    &1    &1    &1    &0.53    &0.02
&0    &0\\

LMC: Accepted $H_{0}$  &0   &0.06  &0.75  &1     &1         &1 &0.52  &0
\\
$V/S$ : Accepted $H_{0}$   &1    &0.97    &0.81    &0.51   &0.30    &0.20
&0.09    &0.05\\
\hline\hline
$\widetilde T_N$: Rejected $H'_{0}$   &1   &1    &0.97   &0.53    &0.02         &0
&0         &0
\\
$\widehat T_{ADG}$: Rejected $H'_{0}$    &1    &1    &0.99    &0.48    &0.01         &0
&0         &0 \\
ADF: Rejected $H'_{0}$     &1    &1    &1    &0.98    &0.60    &0.24
&0.06         &0.01\\
PP : Rejected $H'_{0}$    &1    &1    &1    &1    &0.90    &0.43
&0.05         &0\\
\hline
\hline
\end{tabular}\\
~\\
\vspace{3mm}
\begin{tabular}{|c|c|c|c|c||c|c|c|c|}
\hline\hline
$N=5000$ &&&&&&&&\\
ARFIMA$(0,d,0)$ &$d=-0.2$ &$d=0$ &$d=0.2$ &$d=0.4$ &$d=0.6$ &$d=0.8$ &$d=1$ &$d=1.2$ \\
\hline \hline
$\widetilde S_N$: Accepted $H_{0}$    &1    &1    &1    &1   &0.08    &0
&0         &0 \\
$\widehat S_{ADG}$: Accepted $H_{0}$   &1    &1    &1   &1    &0.06    &0
&0    &0\\
LMC: Accepted $H_{0}$  & 0  &0.05   &0.97   &1     &1        &1 &0.53    &0
\\
$V/S$ : Accepted $H_{0}$   &1    &0.95    &0.50    &0.17    &0.05    &0
&0    &0\\
\hline\hline
$\widetilde T_N$: Rejected $H'_{0}$   &1   &1    &1    &0.94   &0         &0
&0         &0
\\
$\widehat T_{ADG}$: Rejected $H'_{0}$    &1   &1    &1    &0.89    &0         &0
&0         &0 \\
ADF: Rejected $H'_{0}$     &1    &1    &1    &1    &0.88    &0.53
&0.07         &0\\
PP : Rejected $H'_{0}$    &1    &1    &1    &1    &1    &0.75
&0.07         &0\\
\hline
\hline
\end{tabular}
\end{center}
}

\caption{{\small\label{Table6}Comparisons of stationarity and nonstationarity tests from $300$ independent replications of ARFIMA$(0,d,0)$ processes for  several values of $d$ and $N$. The accuracy of tests is measured by the frequencies of  trajectories ``accepted as stationary'' (accepted $H_0$ or rejected $H'_0$) among the $300$ replications which should be close to $1$ for $d\in(-0.5,0.5)$ and close to $0$ for $d\in[0.5,1.2]$}
}
\end{table}

\begin{table}
{\footnotesize
\begin{center}
\begin{tabular}{|c|c|c|c|c||c|c|c|c|}
\hline\hline
$N=500$ &&&&&&&&\\
ARFIMA$(1,d,1)$ &$d=-0.2$ &$d=0$ &$d=0.2$ &$d=0.4$ &$d=0.6$ &$d=0.8$ &$d=1$ &$d=1.2$ \\
$\phi=-0.3$ ; $\theta=0.7$ &&&&&&&&\\
\hline \hline
$\widetilde S_N$: Accepted $H_{0}$    &1    &1   &1    &0.95    &0.47    &0.11
&0.01         &0 \\
$\widehat S_{ADG}$: Accepted $H_{0}$   &1    &1    &1    &0.98    &0.31    &0
&0    &0\\
LMC: Accepted $H_{0}$  &0.12   &0  &0   &0    &0     &0 &0  &0
\\
$V/S$ : Accepted $H_{0}$   &1    &0.96    &0.78    &0.54   &0.34    &0.18
&0.09    &0.05\\
\hline\hline
$\widetilde T_{N}$: Rejected $H'_{0}$   &1   &1    &0.84    &0.23    &0.01         &0
&0         &0
\\
$\widehat T_{ADG}$: Rejected $H'_{0}$    &1    &1    &0.96    &0.21    &0         &0
&0         &0 \\
ADF: Rejected $H'_{0}$     &1    &1    &1    &0.96    &0.59    &0.26
&0.05         &0.01\\
PP : Rejected $H'_{0}$    &1    &1    &1    &1    &0.74    &0.30
&0.03         &0\\
\hline
\hline
\end{tabular}\\
~\\
\vspace{3mm}
\begin{tabular}{|c|c|c|c|c||c|c|c|c|}
\hline\hline
$N=5000$ &&&&&&&&\\
ARFIMA$(1,d,1)$ &$d=-0.2$ &$d=0$ &$d=0.2$ &$d=0.4$ &$d=0.6$ &$d=0.8$ &$d=1$ &$d=1.2$ \\
$\phi=-0.3$ ; $\theta=0.7$ &&&&&&&&\\
\hline \hline
$\widetilde S_N$: Accepted $H_{0}$    &1    &1    &1    &0.99   &0.12    &0
&0         &0 \\
$\widehat S_{ADG}$: Accepted $H_{0}$   &1    &1    &1   &1    &0.04    &0
&0    &0\\
LMC: Accepted $H_{0}$  &0    &0  &0 &0    &0      &0&0   &0
\\
$V/S$ : Accepted $H_{0}$   &1    &0.95    &0.61    &0.22    &0.07    &0
&0.01    &0\\
\hline\hline
$\widetilde T_{N}$: Rejected $H'_{0}$   &1   &1    &1    &0.67   &0.01         &0
&0         &0
\\
$\widehat T_{ADG}$: Rejected $H'_{0}$    &1   &1    &1    &0.86    &0         &0
&0         &0 \\
ADF: Rejected $H'_{0}$     &1    &1    &1    &1    &0.91    &0.45
&0.04         &0\\
PP : Rejected $H'_{0}$    &1    &1    &1    &1    &0.99    &0.59
&0.03         &0\\
\hline
\hline
\end{tabular}
\end{center}
}
\caption{{\small \label{Table7} Comparisons of stationarity and nonstationarity tests from $300$ independent replications of ARFIMA$(1,d,1)$  processes (with $\phi=-0.3$ and $\theta=0.7$) for  several values of $d$ and $N$. The accuracy of tests is measured by the frequencies of  trajectories ``accepted as stationary'' (accepted $H_0$ or rejected $H'_0$) among the $300$ replications which should be close to $1$ for $d\in(-0.5,0.5)$ and close to $0$ for $d\in[0.5,1.2]$}
}
\end{table}

\begin{table}
{\footnotesize
\begin{center}
\begin{tabular}{|c|c|c|c|c||c|c|c|c|}
\hline\hline
$N=500$ &&&&&&&&\\
$X^{(d,\beta)}$ &$d=-0.2$ &$d=0$ &$d=0.2$ &$d=0.4$ &$d=0.6$ &$d=0.8$ &$d=1$ &$d=1.2$ \\
\hline \hline
$\widetilde S_N$: Accepted $H_{0}$    &1    &1   &1    &1    &0.99    &0.49
&0.05         &0.01 \\
$\widehat S_{ADG}$: Accepted $H_{0}$   &1    &1    &1    &1    &0.98    &0.34
&0.01    &0\\
LMC: Accepted $H_{0}$  &0 &0 &0.13   &0.86   &1    &1         &0.82 &0 
\\
$V/S$ : Accepted $H_{0}$   &1    &1    &0.93    &0.69   &0.42    &0.25
&0.16    &0.09\\
\hline\hline
$\widetilde T_{N}$: Rejected $H'_{0}$   &1   &1    &1  &0.93    &0.37         &0
&0         &0
\\
$\widehat T_{ADG}$: Rejected $H'_{0}$    &1    &1    &1    &0.98   &0.23        &0
&0         &0 \\
ADF: Rejected $H'_{0}$     &1    &1    &1    &1    &1    &0.88
&0.38         &0.11\\
PP : Rejected $H'_{0}$    &1    &1    &1    &1    &0.90    &0.43
&0.05         &0\\
\hline
\hline
\end{tabular}\\
~\\
\vspace{3mm}
\begin{tabular}{|c|c|c|c|c||c|c|c|c|}
\hline\hline
$N=5000$ &&&&&&&&\\
$X^{(d,\beta)}$ &$d=-0.2$ &$d=0$ &$d=0.2$ &$d=0.4$ &$d=0.6$ &$d=0.8$ &$d=1$ &$d=1.2$ \\
\hline \hline
$\widetilde S_N$: Accepted $H_{0}$    &1    &1    &1    &1   &1    &0.03
&0         &0 \\
$\widehat S_{ADG}$: Accepted $H_{0}$   &1    &1    &1   &1    &0.99    &0
&0    &0\\
LMC: Accepted $H_{0}$  &0&0   &0.39   &1    &1         &1 &1   &0
\\
$V/S$ : Accepted $H_{0}$   &1    &0.99    &0.79    &0.29    &0.11    &0.04
&0    &0\\
\hline\hline
$\widetilde T_N$: Rejected $H'_{0}$   &1   &1    &1    &0.99   &0.82         &0
&0         &0
\\
$\widehat T_{ADG}$: Rejected $H'_{0}$    &1   &1    &1    &1    &0.30         &0
&0         &0 \\
ADF: Rejected $H'_{0}$     &1    &1    &1    &1    & 1    &0.98
&0.34         &0.01\\
PP : Rejected $H'_{0}$    &1    &1    &1    &1    &1    &1
&0.50         &0.01\\
\hline
\hline
\end{tabular}
\end{center}
}
\caption{{ \small \label{Table8}Comparisons of stationarity and nonstationarity tests from $300$ independent replications of $X^{(d,\beta)}$ processes for several values of $d$ and $N$. The accuracy of tests is measured by the frequencies of  trajectories ``accepted as stationary'' (accepted $H_0$ or rejected $H'_0$) among the $300$ replications which should be close to $1$ for $d\in(-0.5,0.5)$ and close to $0$ for $d\in[0.5,1.2]$}
}
\end{table}
~\\
\underline{{\bf Conclusions of simulations:}} From their constructions, KPSS and LMC, V/S (or KPSS), ADF and PP tests should asymptotically decide the stationarity hypothesis when $d=0$, and the nonstationarity hypothesis when $d>0$. It was exactly what we observe in these simulations. For ARIMA$(p,d,0)$ processes with $d=0$ or $d=1$ ({\it i.e.} AR$(1)$ process when $d=0$), LMC, V/S, ADF and PP  tests are more accurate than our adaptive MIR tests or tests based on $\widehat d_N^{ADG}$, especially when $N=500$. But when $N=5000$ the tests computed from $\widetilde d_N^{IR}$ and $\widehat d_N^{ADG}$ provide however convincing results.\\
In case of processes with $d\in (0,1)$, the tests computed from $\widetilde d_N^{IR}$ and $\widehat d_N^{ADG}$ are clearly better performances than than classical stationarity tests ADF or PP which accept the nonstationarity assumption $H_0'$ even if the processes are stationary when $0<d<0.5$ for instance. The results obtained with the LMC test are not at all satisfying even when another lag parameter is chosen. The case of the V/S test is different since this test is built for distinguishing between short and long memory processes. Note that a renormlized version of this test has been defined in Giraitis {\it et al.} (2006) for also taking account of the value of $d$.

\subsection{Application to the the Stocks and the Exchange Rate Markets}

We applied the adaptive MIR statistics as well as the other famous long-memory estimators and stationarity tests to Econometric data, the Stocks and Exchange Rate Markets. More precisely, the $5$ following daily closing value time series are considered:
\begin{enumerate}
\item The USA Dollar Exchange rate in Deusch-Mark, from $11/10/1983$ to $08/04/2011$ ($7174$ obs.).
\item The USA Dow Jones Transportation Index, from $31/12/1964$ to $08/04/2011$ ($12072$ obs.).
\item The USA Dow Jones Utilities Index, from $31/12/1964$ to $08/04/2011$ ($12072$ obs.).
\item The USA Nasdaq Industrials Index, from $05/02/1971$ to $08/04/2011$ ($10481$ obs.).
\item The Japan Nikkei225A Index, from $03/04/1950$ to $8/04/2011$ ($15920$ obs.).
\end{enumerate}
We considered the log-return of this data and tried to test their stationarity properties. Since stationarity or nonstationarity tests are not able to detect (offline) changes, we first used an algorithm developed by M. Lavielle for detecting changes (this free software can be downloaded from his homepage: {\tt http://www.math.u-psud.fr/$\sim$lavielle/programmes$\overline{~}$lavielle.html}). This algorithm provides the choice of detecting changes in mean, in variance, ..., and we chose to detect parametric changes in the distribution. Note that the number of changes is also estimated since this algorithm is based on the minimization of a penalized contrast. We obtained for each time series an estimated number of changes equal to $2$ which are  the following:
\begin{itemize}
\item Two breaks points for the US dollar-Deutsch Mark Exchange rate return are estimated, corresponding to the dates: 21/08/2006 and 24/12/2007.
The Financial crisis of 2007-2011, followed by the late 2000s recession and the 2010 European sovereign debt crisis can cause
such breaks.
\item Both the breaks points estimated for the US Dow Jones Transportation Index return, of the New-York Stock Market, correspond to the dates: 17/11/1969 and 15/09/1997. The first break change can be a consequence on transportation companies difficulties the American Viet-Nam war against communist block. The second change point can be viewed as a contagion by the spread of the Thai crisis in 1997 to other countries and mainly the US stock Market.
\item Both the breaks points estimated for the US Dow Jones Utilities Index return correspond to the dates: 02/06/1969 and 14/07/1998.
    The same arguments as above can justify the first break. The second at 1998 is probably a consequence of ``\emph{the long very acute crisis in the bond markets,..., the dramatic fiscal crisis and Russian Flight to quality caused by it, may have been warning the largest known by the global financial system: we never went too close to a definitive breakdown of relations between the various financial instruments}``(Wikipedia).
\item The two breaks points for the  US Nasdaq Industrials Index return correspond to the  dates: 17/07/1998 and 27/12/2002. The first break at 1998 is explained by the Russian flight to quality as above. The second break at 2002 corresponds to the Brazilian public debt crisis of 2002 toward foreign owners (mainly the U.S. and the IMF) which implicitly assigns a default of payment probability close to 100\%  with a direct impact on the financial markets indexes as the Nasdaq.
\item Both the breaks points estimated for the Japanese Nikkei225A Index return corresponds to the dates 29/10/1975 and 12/02/1990, perhaps as consequence of the strong dependency of Japan to the middle east Oil following 1974 or anticipating 1990 oil crisis. The credit crunch which is seen as a major factor in the U.S. recession of 1990-91 can play a role in the second break point.
\end{itemize}
Data and estimated instant breaks can be seen on Figure \ref{ThreeRegimes}. Then, we applied the estimators and tests described in the previous subsection on trajectories obtained in each stages for the $5$ economic time series. These applications were done on the log-returns, their absolute values, their squared values and their $\theta$-power laws with $\theta$ maximized for each LRD estimators. The results of these numerical experiments can be seen in Tables \ref{Table11}-\ref{Table15}. \\
~\\
\underline{{\bf Conclusions of numerical experiments:}} We exhibited again the well known result: the log-returns are stationary and short memory processes while absolute values or power $\theta$ of log-returns are generally stationary but long memory processes (for this conclusion, we essentially consider the results of $\widetilde S_N$, $\widetilde T_N$ and V/S tests since the other tests have been shown not to be relevant in the cases of long-memory processes). However the last and third estimated stage of each time series provides generally the largest estimated values of the memory parameter $d$ (for power law of log-returns) which are close to $0.5$; hence, for Nasdaq time series, we accepted the nonstationarity assumption.
\begin{figure}
\includegraphics[height=3.9cm,width=8cm]{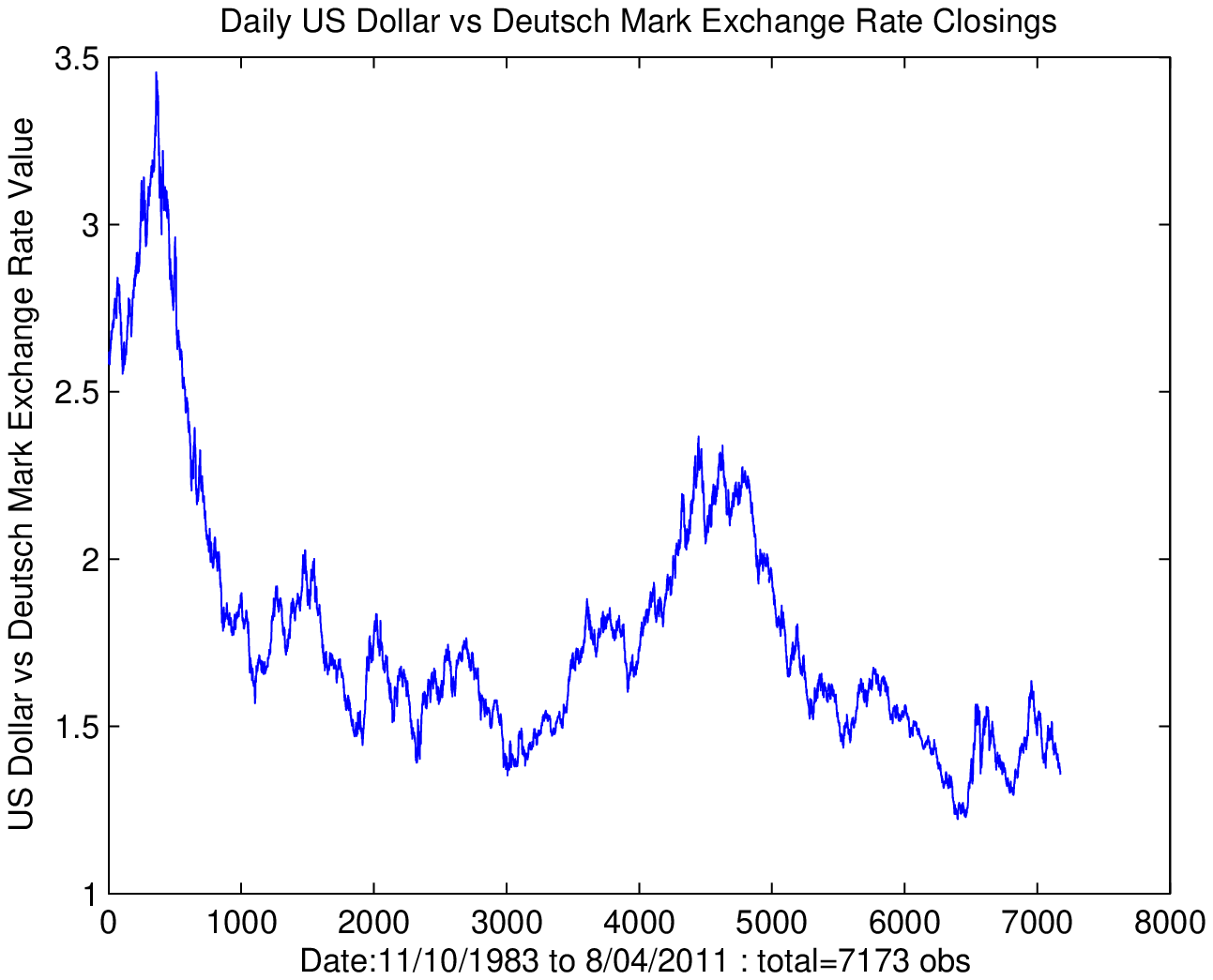}\hspace*{-1.3cm}
\includegraphics[height=3.9cm,width=7cm]{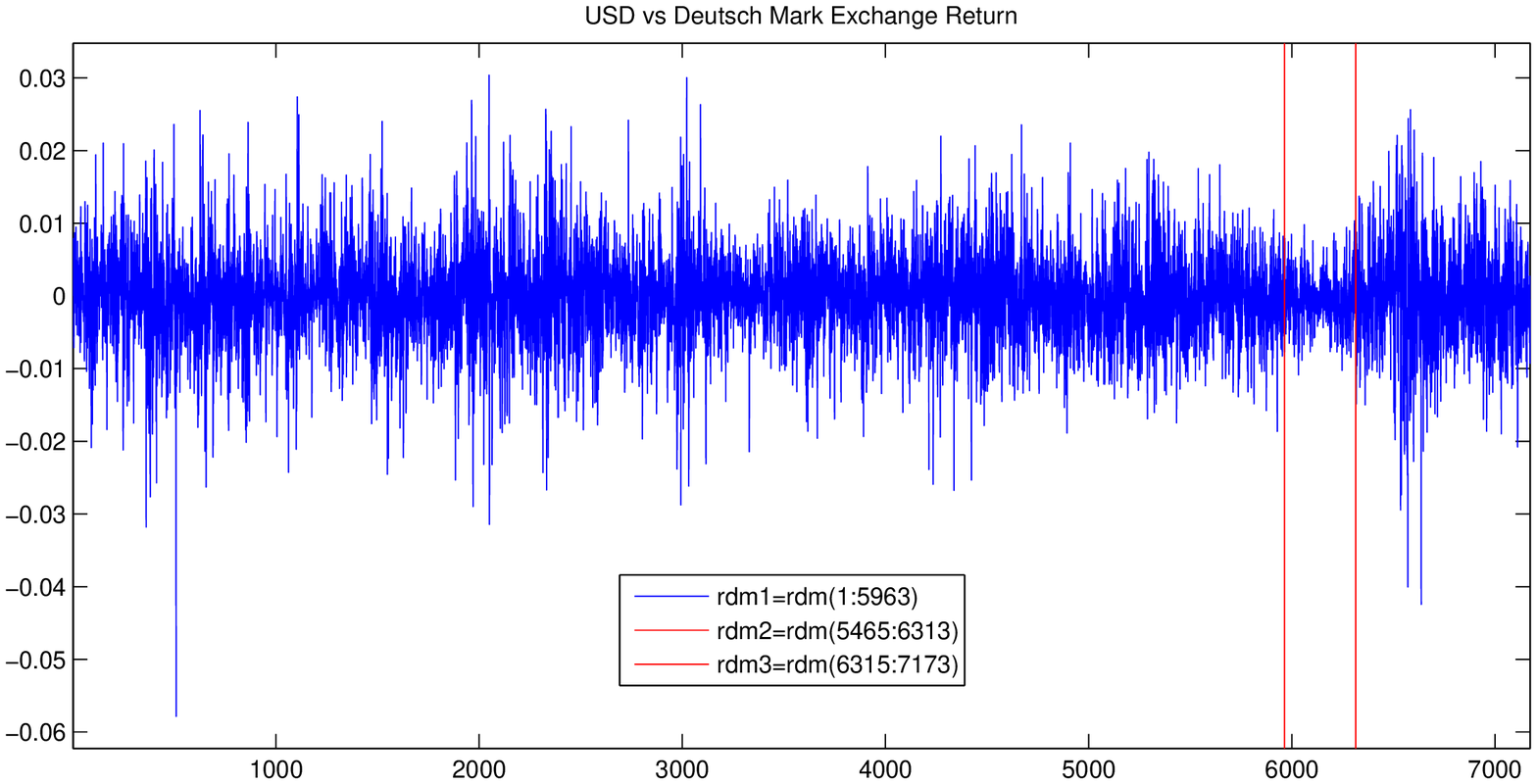}
\includegraphics[height=3.9cm,width=8cm]{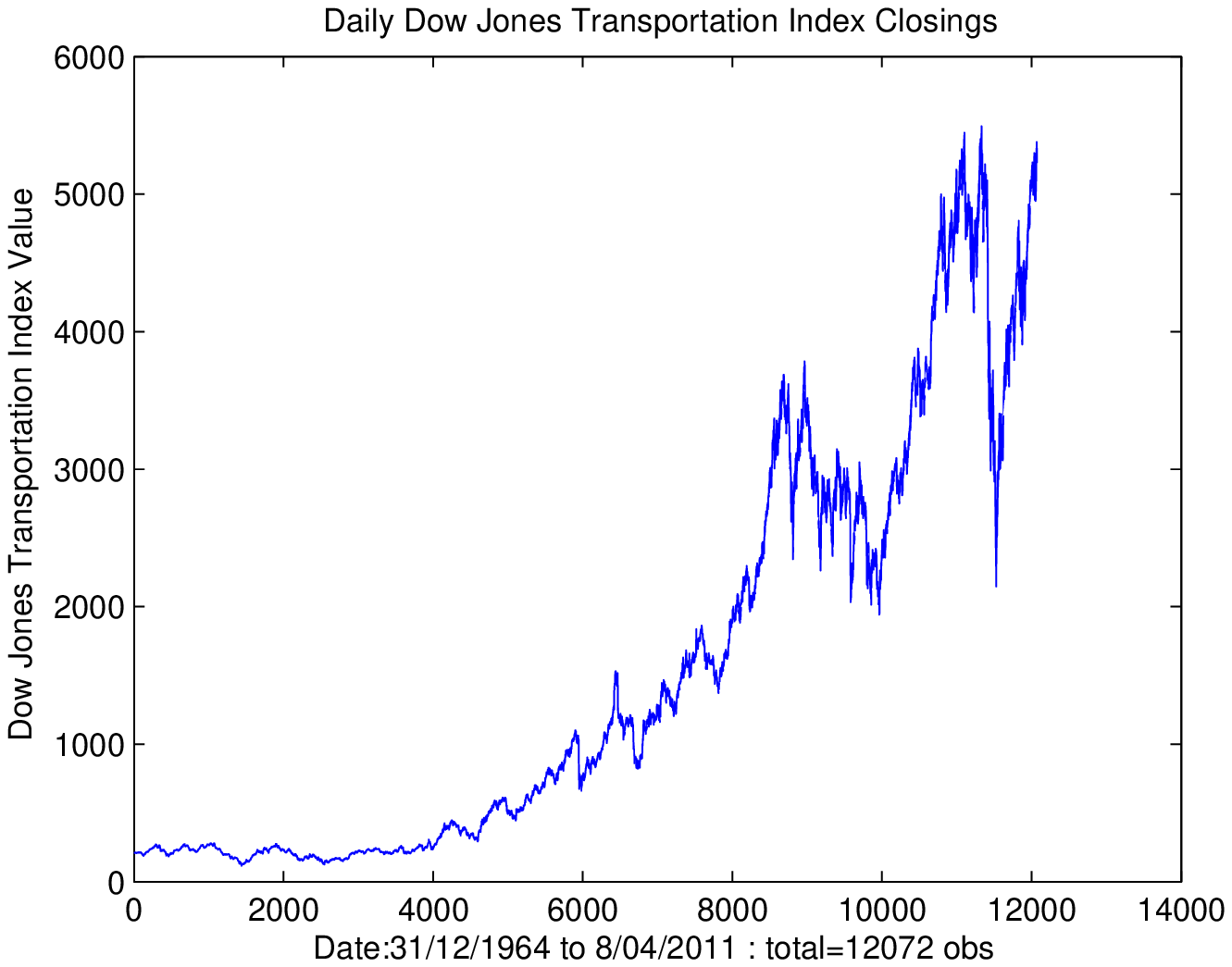}\hspace{-0.1cm}
\includegraphics[height=3.9cm,width=8cm]{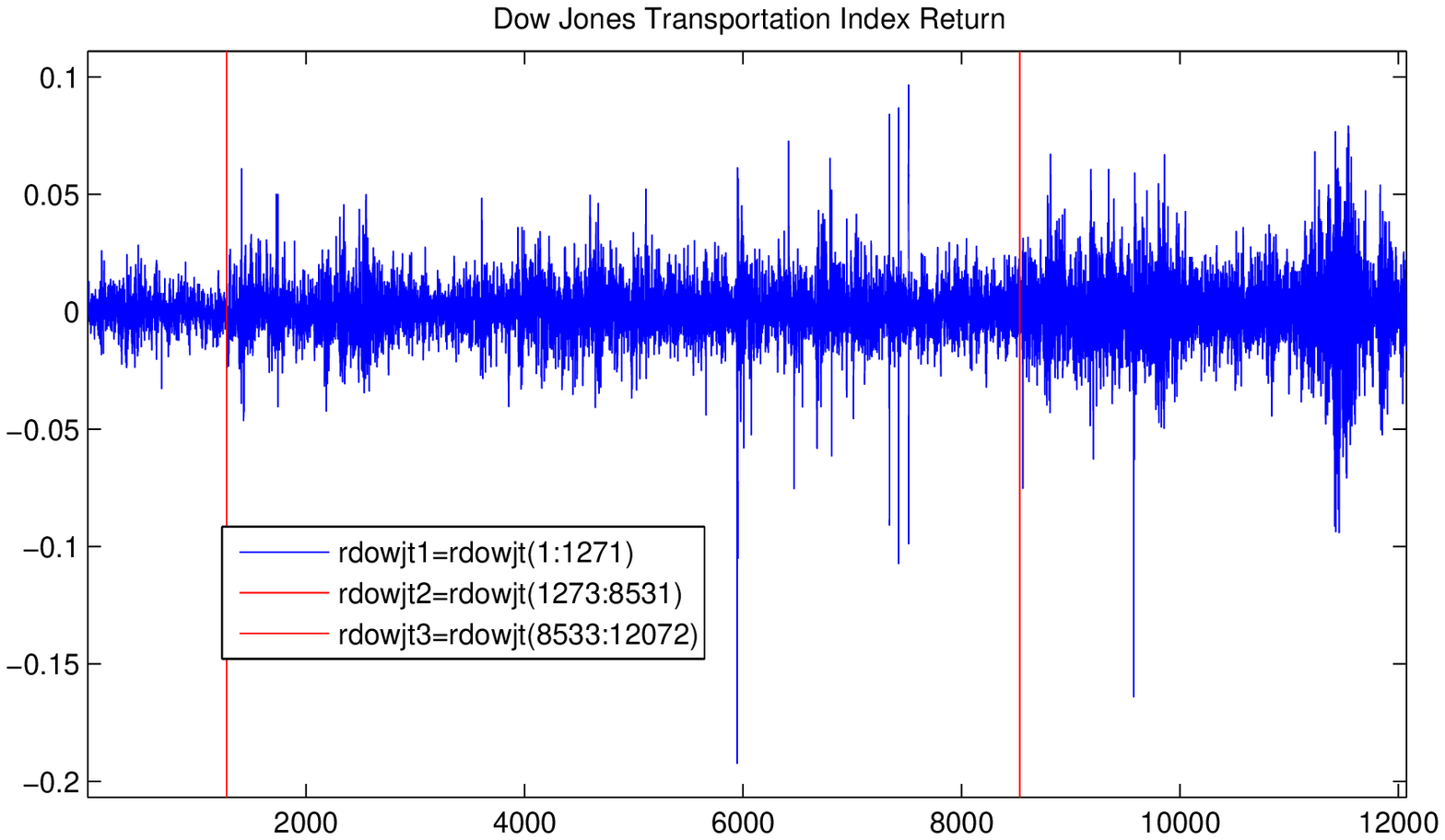}
\includegraphics[height=3.9cm,width=8cm]{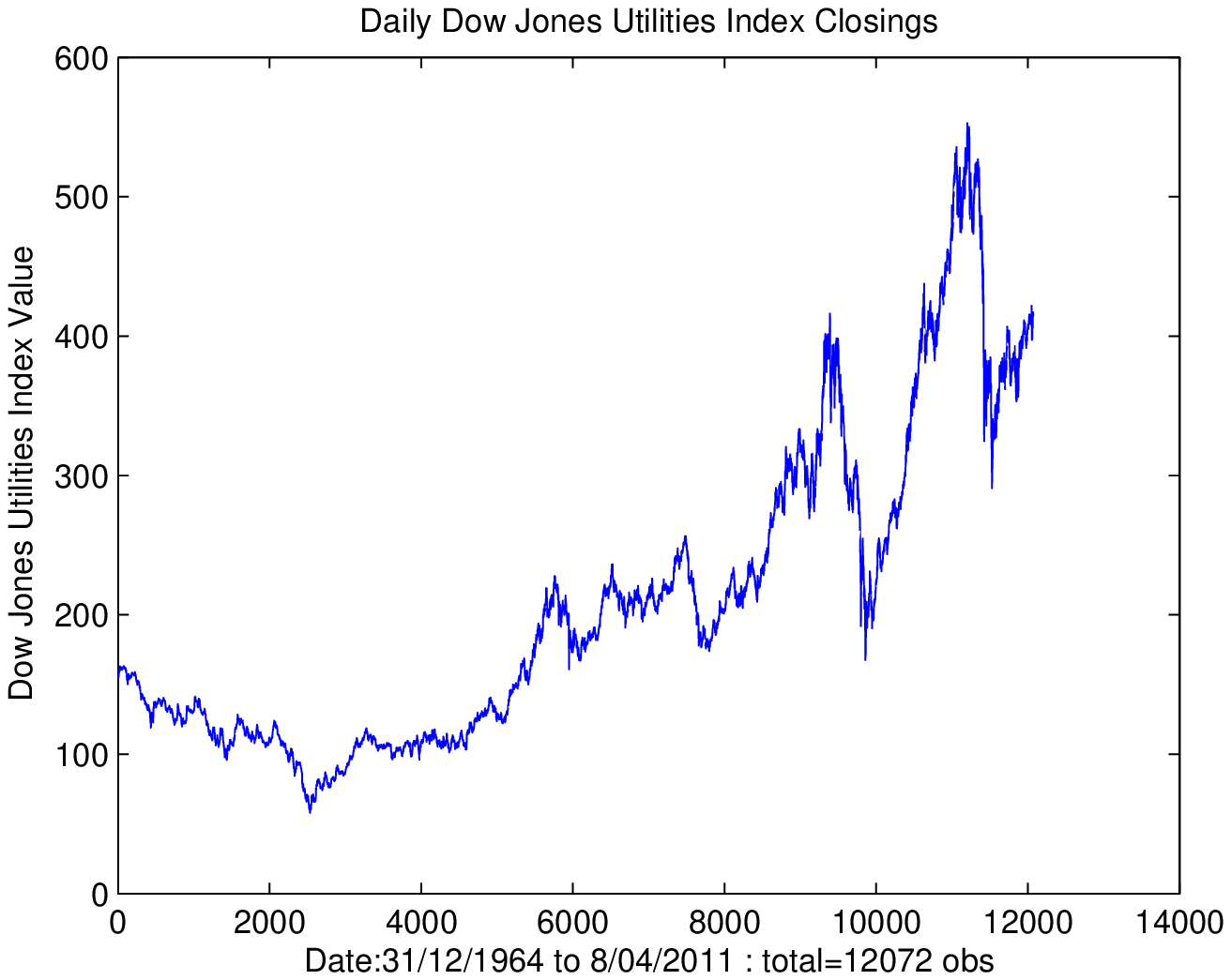}\hspace{-0.1cm}
\includegraphics[height=3.9cm,width=8cm]{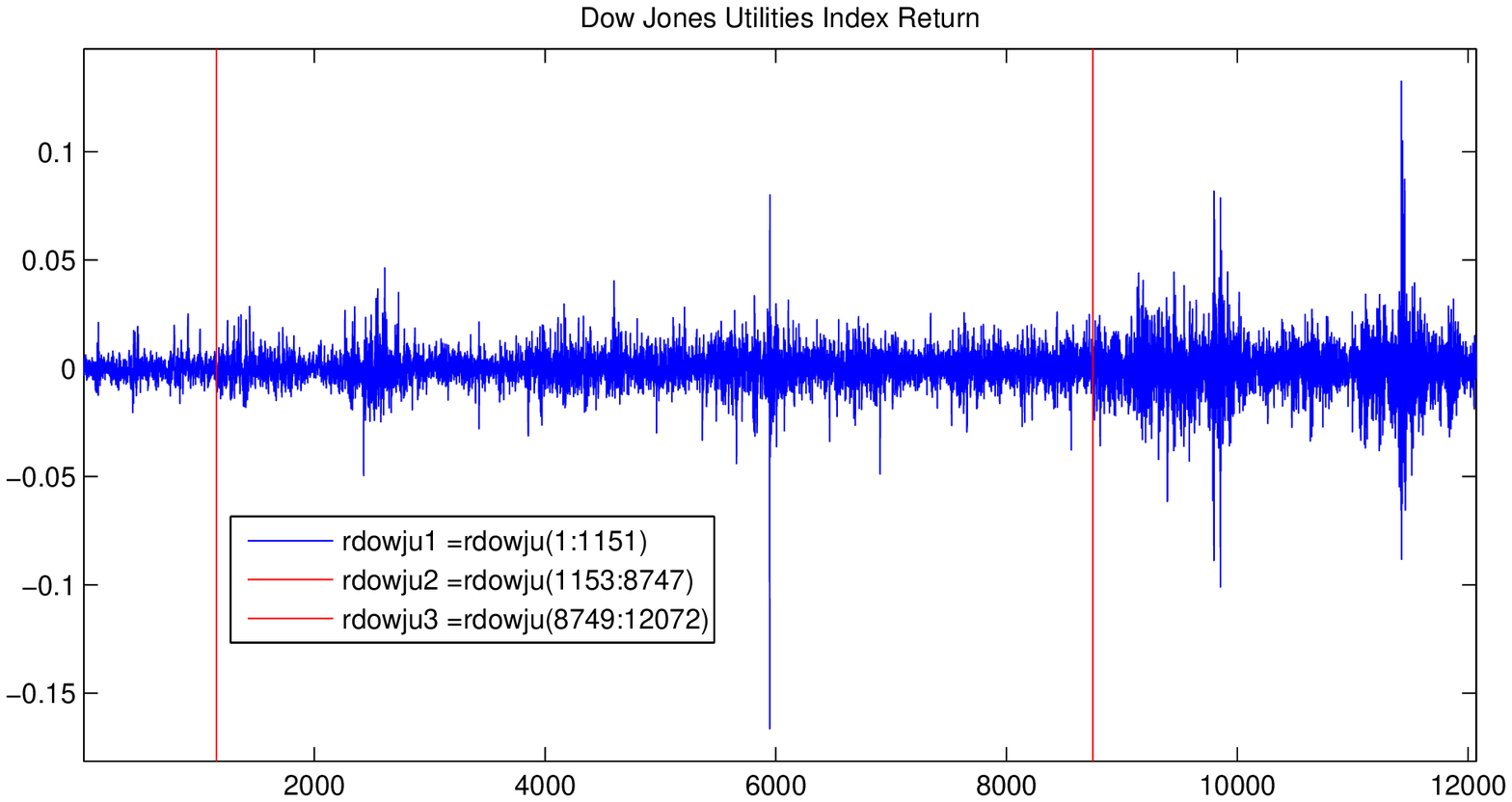}
\includegraphics[height=3.9cm,width=8cm]{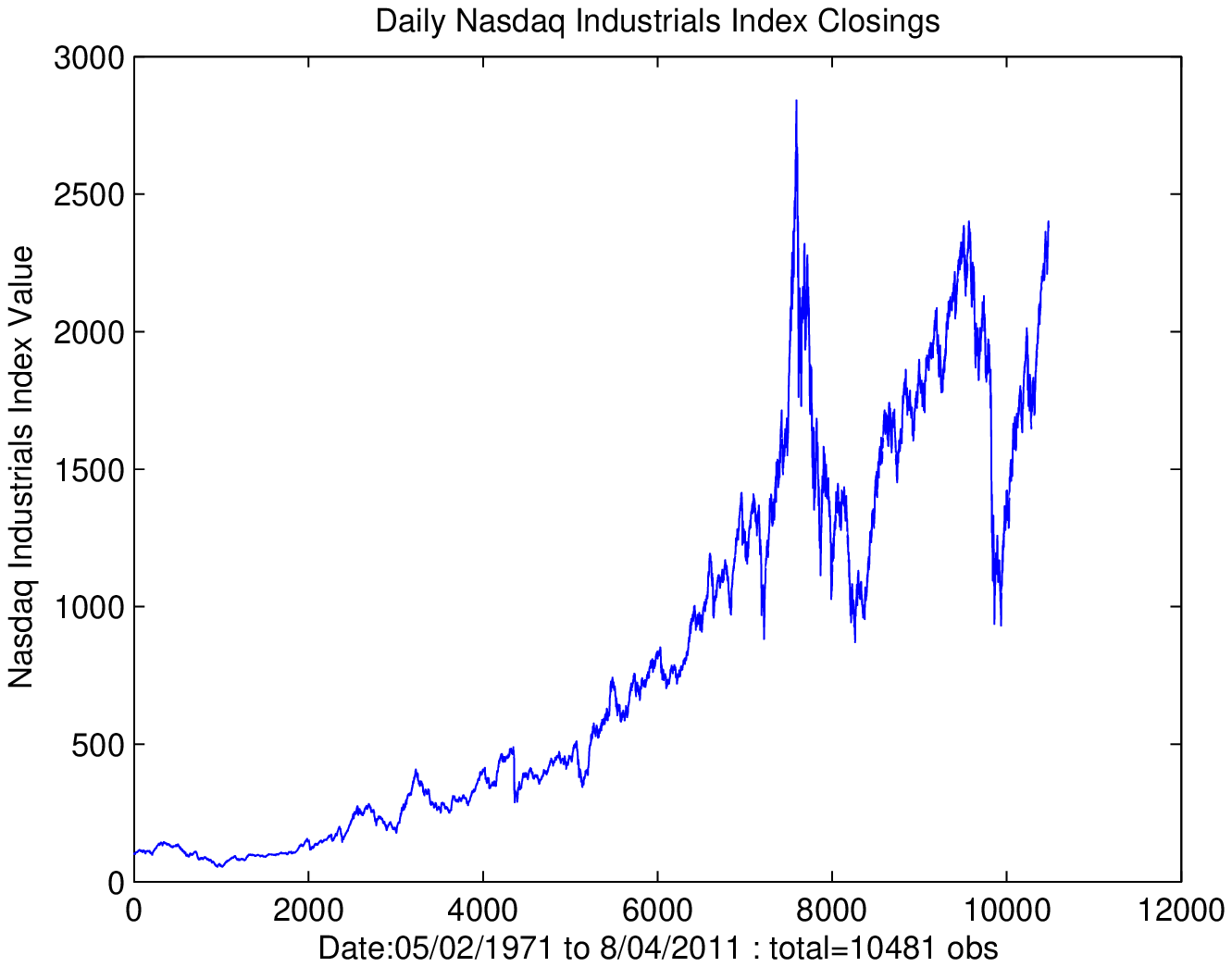}\hspace{-0.1cm}
\includegraphics[height=3.9cm,width=8cm]{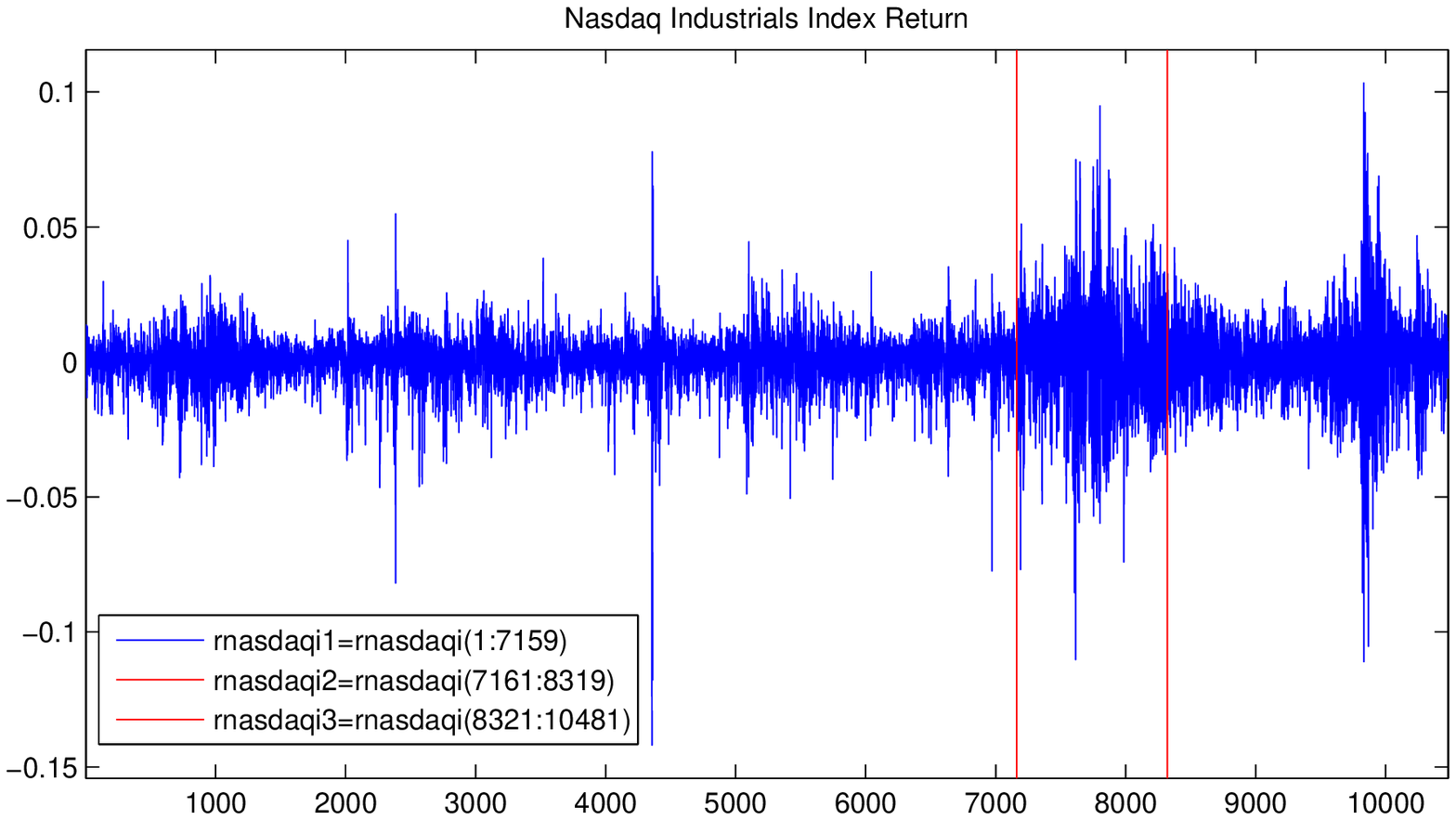}
\includegraphics[height=3.9cm,width=8cm]{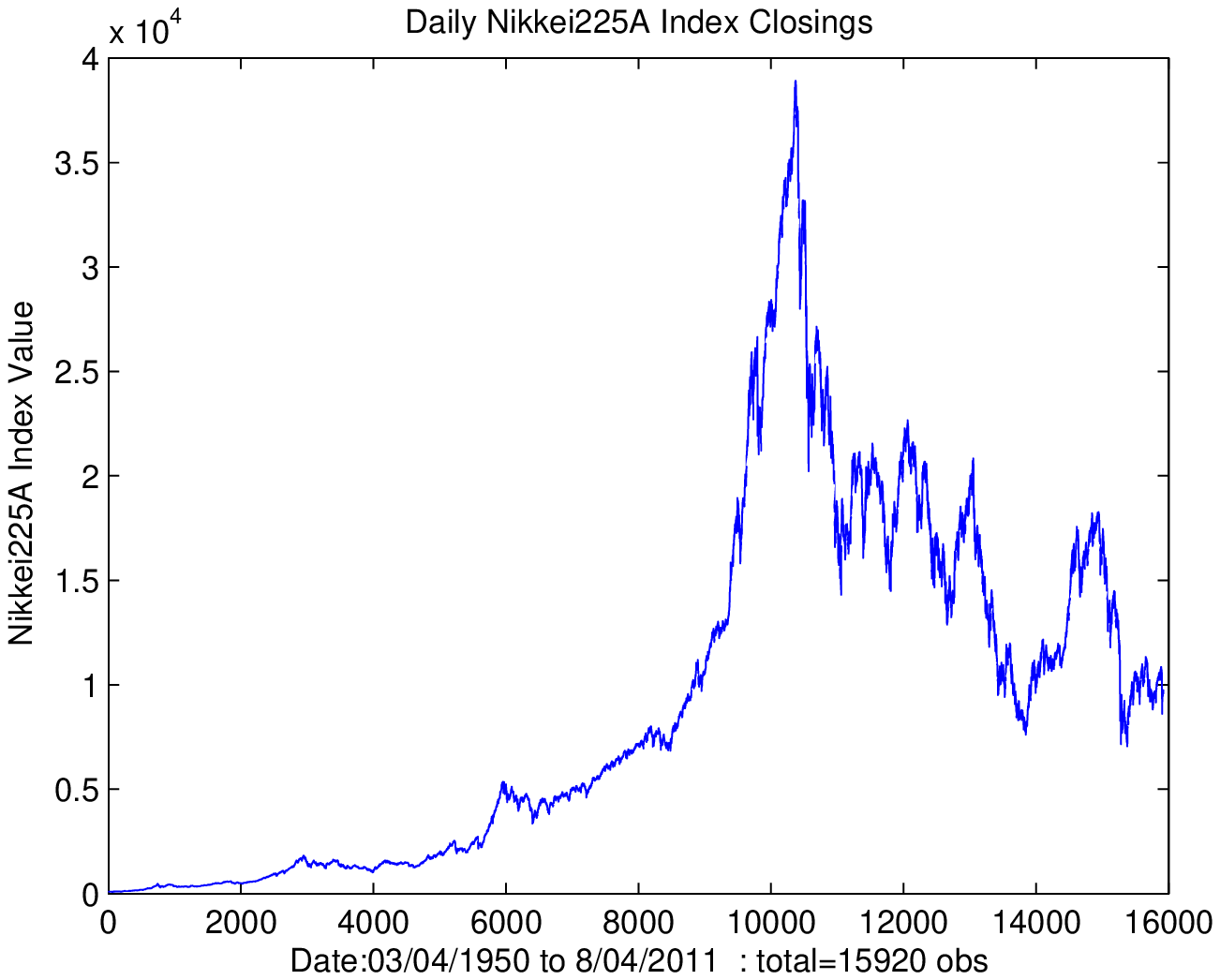}\hspace{0.4cm}
\includegraphics[height=3.9cm,width=8cm]{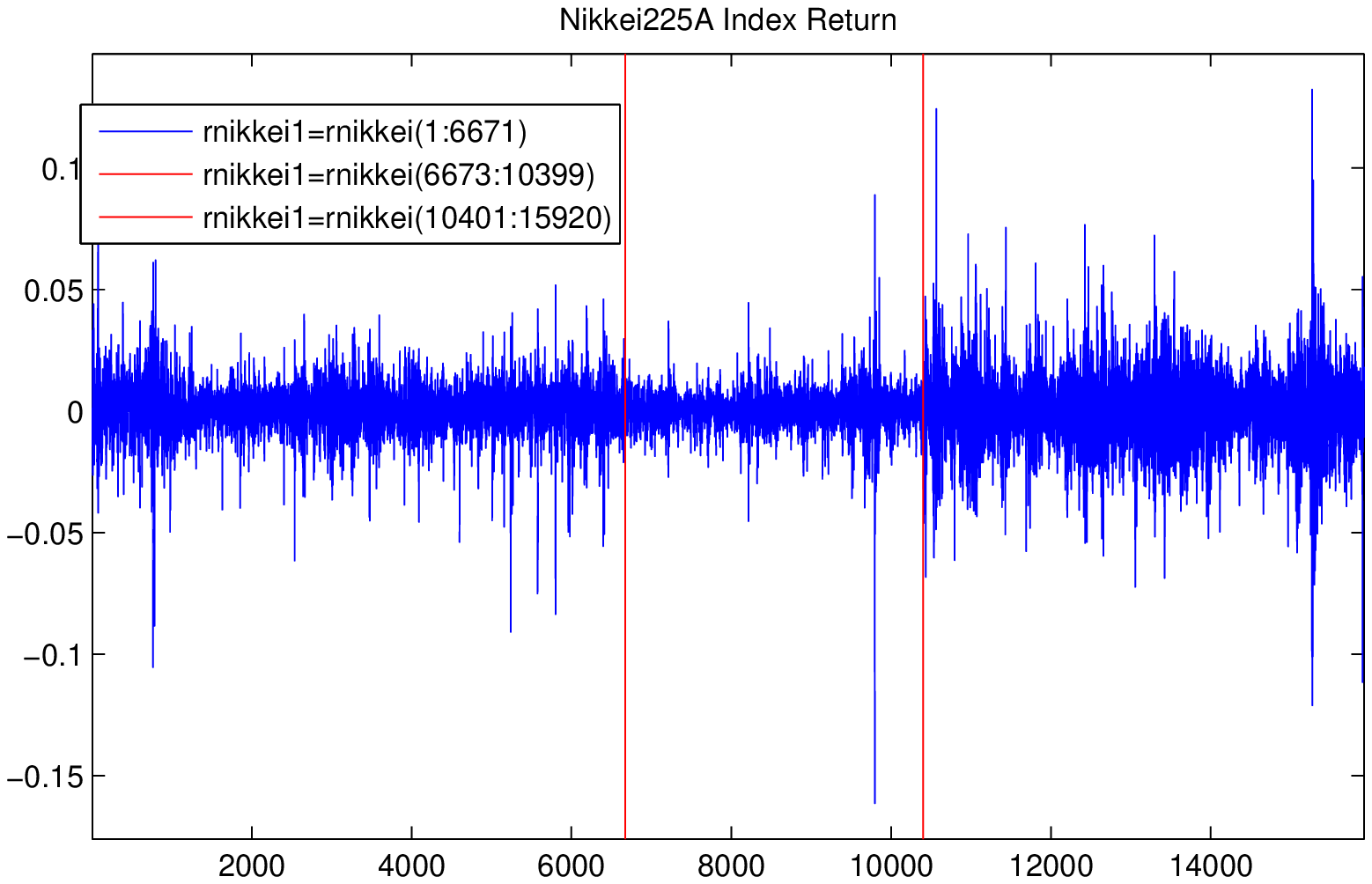}
\caption{The financial data (DowJonesTransportations, DowJonesUtilities, NasdaqIndustrials, Nikkei225A and US Dollar vs Deutsch Mark): original data (left) and log-return with their both estimated breaks instants (right) occurred at the distribution Changes}\label{ThreeRegimes}
\end{figure}

\begin{table}[t11]
{\footnotesize
%{\scriptsize
%{\tiny
\begin{center}
r=(USD1 vs Deutsh-Mark Exchange Rate Return)
\begin{tabular}{|c||c|c|c|c|c|c||c||c|c||c|c|c|c||}
\hline
\hline

Segments & \multicolumn{6}{c||}{(Non)Stationarity Test}& LRD & Kurtosis&Skewness &\multicolumn{4}{c||}{$\widehat {d}$}
\\
\cline{2-14}
Breaks&$\widetilde S_N$ & $\widetilde T_N$  & ADF & PP & KPSS & LMC &$V/S$ & $\widetilde{\kappa}$ & $\widetilde{s}$ &$\widetilde{d^{IR}}$ &$\widetilde{d^{MS}}$ &$\widehat{d^{ADG}}$ &$\widetilde{d^{WAV}}$
\\
\hline
$[1:5963]$ &S    &S    &S    &S     &S     &S  &SM    &5.5   &-0.2   &-0.031    &0.059   &0.057   &-0.007\\
\hline
$[5965:6313]$ &S    &S    &S    &S     &S     &S  &SM      &3.4    &0.1    &0.034    &0.169   &0.122   &-0.015\\
\hline
$[6315:7173]$  &S    &S    &S    &S     &S     &S  &SM      &5.3   &-0.4    &0.098    &0.140   &0.043    &0.019
\\
\hline \hline
\end{tabular}
~\\
\vspace{3mm}
$|r|=abs(\mbox{USD1 vs Deutsh-Mark Exchange Rate Return})$
\begin{tabular}{|c||c|c|c|c|c|c||c||c|c||c|c|c|c||}
\hline
\hline
Segments & \multicolumn{6}{c||}{(Non)Stationarity Test}& LRD & Kurtosis&Skewness &\multicolumn{4}{c||}{$\widehat {d}$}
\\
\cline{2-14}
Breaks&$\widetilde S_N$ & $\widetilde T_N$  & ADF & PP & KPSS & LMC &$V/S$ & $\widetilde{\kappa}$ & $\widetilde{s}$ &$\widetilde{d^{IR}}$ &$\widetilde{d^{MS}}$ &$\widehat{d^{ADG}}$ &$\widetilde{d^{WAV}}$
\\
\hline
$[1:5963]$ &S    &S    &S     &S     &NS     &NS     &LM    &9.5    &1.8    &0.294    &0.301    &0.344    &0.275\\
\hline
$[5965:6313]$ &S    &S    &NS     &S     &NS     &NS     &LM    &3.6  &1.1   &-0.121    &0.153    &0.414   &-0.038\\
\hline
$[6315:7173]$ &S    &S    &S     &S     &NS    &NS      &LM    &9.2    &1.8    &0.168    &0.417    &0.389    &0.410
\\
\hline \hline
\end{tabular}
~\\
\vspace{3mm}
$r^{2}=(\mbox{USD1 vs Deutsh-Mark Exchange Rate Return})^{2}$
\begin{tabular}{|c||c|c|c|c|c|c||c||c|c||c|c|c|c||}
\hline
\hline
Segments & \multicolumn{6}{c||}{(Non)Stationarity Test}& LRD & Kurtosis&Skewness &\multicolumn{4}{c||}{$\widehat {d}$}
\\
\cline{2-14}
Breaks&$\widetilde S_N$ & $\widetilde T_N$  & ADF & PP & KPSS & LMC &$V/S$ & $\widetilde{\kappa}$ & $\widetilde{s}$ &$\widetilde{d^{IR}}$ &$\widetilde{d^{MS}}$ &$\widehat{d^{ADG}}$ &$\widetilde{d^{WAV}}$
\\
\hline
$[1:5963]$ &S    &S &S     &S    &S     &S   &LM  &289.5   &10.7    &0.081    &0.258   &0.298    &0.078\\
\hline
$[5965:6313]$ &S   &S   &S     &S     &NS     &S   &LM    &8.7    &2.3   &-0.018    &0.127   &0.431   &-0.096\\
\hline
$[6315:7173]$ &S    &S &S    &S     &NS     &S   &LM   &81.3    &7.1    &0.035    &0.411   &0.336    &0.428
\\
\hline \hline
\end{tabular}
~\\
\vspace{3mm}
$|r|^{\theta}=(abs(\mbox{USD1 vs Deutsh-Mark Exchange Rate Return}))^{\theta}$
\begin{tabular}{|c||c|c|c|c|c|c||c||c|c||c|c|c|c||}
\hline
\hline

$\widehat{\theta}_{i}^{(j)}=$ & \multicolumn{6}{c||}{(Non)Stationarity Test}& LRD& Kurtosis&Skewness &\multicolumn{4}{c||}{$\widehat{d}$}
\\
\cline{2-14}
$\mbox{ArgMax}_\theta(\widehat{d}(|r_{i}|^{\theta}))$&$\widetilde S_N$ & $\widetilde T_N$   & ADF & PP & KPSS & LMC &$V/S$ & $\widetilde{\kappa}$ & $\widetilde{s}$ &$\widetilde{d^{IR}}$ &$\widetilde{d^{MS}}$ &$\widehat{d^{ADG}}$ &$\widetilde{d^{WAV}}$
\\
\hline
$\widehat{\theta}_{1}^{IR}$=0.32 &S    &S    &NS    &S        &NS&NS  &SM    &3.5   &-0.5    &\textbf{0.321}*  &0.251    &0.256    &0.343\\
$\widehat{\theta}_{1}^{MS}$= 0.97&S    &S    &S    &S      &NS    &NS &SM    &8.7   &1.1    &0.293 &\textbf{0.301}*    &0.343    &0.275\\
$\widehat{\theta}_{1}^{ADG}$= 1.12 &S    &S   &S    &S      &NS    &NS   &SM  &13.7    &2.3    &0.302 &0.300    &\textbf{0.345}*    &0.273\\
$\widehat{\theta}_{1}^{WAV}$=0.77 & S    &S    &S   &S      &NS    &NS  &SM  &5.1    &1.1    &0.273  &0.298   &0.335    &\textbf{0.379}*\\
\hline
$\widehat{\theta}_{2}^{IR}$=0.05 &S    &S    &NS   &NS        &S    &NS  &LM   &27.9   &-5.0    &\textbf{0.246}*  &0.078   &-0.005    &0.103\\
$\widehat{\theta}_{2}^{MS}$= 1.31& S    &S   &S    &S      &NS    &NS  &LM   &4.9    &1.5   &-0.103  &\textbf{0.166}*    &0.446   &-0.072\\
$\widehat{\theta}_{2}^{ADG}$=1.50 & S    &S   &S   &S      &NS    &S   &LM  &5.8    &1.8   &-0.092  &0.162    &\textbf{0.450}*   &-0.082\\
$\widehat{\theta}_{2}^{WAV}$=0.03 & S    &S   &NS    &NS   &S    &NS   &LM&30.9   &-5.4    &0.239 &0.113   &-0.030    &\textbf{0.211}*\\
\hline
$\widehat{\theta}_{3}^{IR}$=0.63 & S    &S   &NS   &NS    &NS     &NS   &LM  &3.8    &0.7    &\textbf{0.244}* &0.354    &0.333    &0.097
\\
$\widehat{\theta}_{3}^{MS}$= 1.44 & S    &S   &S   &S    &NS   &NS   &LM   &27.1    &3.7   &0.159 &\textbf{0.436}*    &0.387    &0.441
\\
$\widehat{\theta}_{3}^{ADG}$= 1.19& S    &S  &S    &S &NS  &NS   &LM  &14.9    &2.6   &0.168  &0.430    &\textbf{0.394}*    &0.430\\
$\widehat{\theta}_{3}^{WAV}$= 2.90& S    &S  &S & S &NS    &S   &LM &223.4   &13.3    &0.053  &0.291    &0.233    &\textbf{0.475}*
\\
\hline
\hline
\end{tabular}
~\\
\end{center}
}
\caption{{\small \label{Table11} Results of stationarity, nonstationarity and V/S tests and the $4$ long memory parameter estimators applied to several functionals $f$ of USD1 vs Deutsh-Mark Exchange Rate Return: from the top to bottom, $f(x)=x$, $f(x)=|x|$, $f(x)=x^2$ and $f(x)=|x|^\theta$ with $\theta$ maximizing the $4$ different long memory parameter estimators (''$S$`` for ''stationarity`` decision and ''$NS$`` for ''nonstationarity`` decision). Statistics are applied to the $3$ estimated stages of each trajectory (obtained from a change detection algorithm). }
}
\end{table}

\begin{table}[t12]
%{\tiny
%{\scriptsize
{\footnotesize
%{\small

\begin{center}
r=Dow Jones Transportation Index Return
\begin{tabular}{|c||c|c|c|c|c|c||c||c|c||c|c|c|c||}
\hline
\hline
Segments & \multicolumn{6}{c||}{(Non)Stationarity Test}& LRD & Kurtosis&Skewness &\multicolumn{4}{c||}{$\widehat {d}$}
\\
\cline{2-14}
Breaks&$\widetilde S_N$ & $\widetilde T_N$  & ADF & PP & KPSS & LMC &$V/S$ & $\widetilde{\kappa}$ & $\widetilde{s}$ &$\widetilde{d^{IR}}$ &$\widetilde{d^{MS}}$ &$\widehat{d^{ADG}}$ &$\widetilde{d^{WAV}}$
\\
\hline
$[1:1271]$ &S   &S  &S     &S    &S    &NS   &SM    &4.6   &0.0    &0.218    &0.174    &0.098    &0.198\\
\hline
$[1273:8531]$ &S    &S  &S     &S     &S    &S  &SM   &21.7   &-0.8    &0.053    &0.002    &0.008   &-0.404\\
\hline
$[8533:12071]$ &S   &S   &S    &S    &S   &S  &SM    &8.3   &-0.3    &0.002   &-0.015    &-0.034   &-0.038
\\
\hline \hline
\end{tabular}
~\\
\vspace{3mm}
$|r|=abs(\mbox{Dow Jones Transportation Index Return})$
\begin{tabular}{|c||c|c|c|c|c|c||c||c|c||c|c|c|c||}
\hline
\hline

Segments & \multicolumn{6}{c||}{(Non)Stationarity Test}& LRD & Kurtosis&Skewness &\multicolumn{4}{c||}{$\widehat {d}$}
\\
\cline{2-14}
Breaks&$\widetilde S_N$ & $\widetilde T_N$  & ADF & PP & KPSS & LMC &$V/S$ & $\widetilde{\kappa}$ & $\widetilde{s}$ &$\widetilde{d^{IR}}$ &$\widetilde{d^{MS}}$ &$\widehat{d^{ADG}}$ &$\widetilde{d^{WAV}}$
\\
\hline
$[1:1272]$ &S   &S   &S   &S   &NS    &NS    &LM    &6.1  &1.5    &0.154    &0.320    &0.270    &0.166\\
\hline
$[1273:8532]$ &S   &S  &S   &S  &NS   &NS     &LM   &57.3    &4.4    &0.322    &0.260    &0.240    &0.168\\
\hline
$[8533:12071]$ &S   &S &S    &S   &NS    &NS   &LM  &16.3    &2.5    &0.405    &0.476    &0.496    &0.374
\\
\hline \hline
\end{tabular}
~\\
\vspace{3mm}
$r^{2}=(\mbox{Dow Jones Transportation Index Return})^{2}$
\begin{tabular}{|c||c|c|c|c|c|c||c||c|c||c|c|c|c||}
\hline
\hline

Segments & \multicolumn{6}{c||}{(Non)Stationarity Test}& LRD & Kurtosis&Skewness &\multicolumn{4}{c||}{$\widehat {d}$}
\\
\cline{2-14}
Breaks&$\widetilde S_N$ & $\widetilde T_N$  & ADF & PP & KPSS & LMC &$V/S$ & $\widetilde{\kappa}$ & $\widetilde{s}$ &$\widetilde{d^{IR}}$ &$\widetilde{d^{MS}}$ &$\widehat{d^{ADG}}$ &$\widetilde{d^{WAV}}$
\\
\hline
$[1:1272]$ &S  &S &S     &S    &NS    &S  &LM  &32.3     &4.4    &0.158   &0.284   &0.231   &0.231\\
\hline
$[1273:8532]$ &S   &S &S   &S   &S     &NS   &LM   &2301.5   &39.9   &0.334   &0.122   &0.093   &0.118\\
\hline
$[8533:12071]$ &S  &NS  &S     &S     &NS    &NS    &LM   &459.0    &15.5  &0.416   &0.452   &0.434   &0.356\\
\hline \hline
\end{tabular}
~\\
\vspace{3mm}
$|r|^{\theta}=(abs(\mbox{Dow Jones Transportation Index Return}))^{\theta}$
\begin{tabular}{|c||c|c|c|c|c|c||c||c|c||c|c|c|c||}
\hline
\hline

$\widehat{\theta}_{i}^{(j)}=$ & \multicolumn{6}{c||}{(Non)Stationarity Test}& LRD& Kurtosis&Skewness &\multicolumn{4}{c||}{$\widehat{d}$}
\\
\cline{2-14}
$\mbox{ArgMax}_\theta(\widehat{d}(|r_{i}|^{\theta})$&$\widetilde S_N$ & $\widetilde T_N$   & ADF & PP & KPSS & LMC &$V/S$ & $\widetilde{\kappa}$ & $\widetilde{s}$ &$\widetilde{d^{IR}}$ &$\widetilde{d^{MS}}$ &$\widehat{d^{ADG}}$ &$\widetilde{d^{WAV}}$
\\
\hline
$\widehat{\theta}_{1}^{IR}$= 1.83    &S    &S  &S    &S   &NS     &S  &LM   &25.0    &3.8    &\textbf{0.252}*   &0.291    &0.237    &0.202\\
$\widehat{\theta}_{1}^{MS}$= 0.45    &S    &S   &NS   &S   &NS     &NS   &LM    &2.7    &0.4    &0.118   &\textbf{0.331}*    &0.290    &0.047\\
$\widehat{\theta}_{1}^{ADG}$= 0.36    &S   &S  &NS    &S    &NS    &NS   &LM    &2.9   &-0.4    &0.118   &0.329    &\textbf{0.291}*    &0.237\\
$\widehat{\theta}_{1}^{WAV}$= 0.03    &S    &S  &NS    &S    &NS      &NS    &LM   &12.8   &-3.4    &0.149   &0.257    &0.260    &\textbf{0.327}*\\
\hline
$\widehat{\theta}_{2}^{IR}$= 2.06   &S   &S   &S     &S    &S    &NS  &LM   &2551.6      &42.6   &\textbf{0.355}*   &0.113   &0.086   &0.110\\
$\widehat{\theta}_{2}^{MS}$= 0.68   &S   &S   &S   &S   &NS   &NS  &LM     &10.6       &1.6   &0.308   &\textbf{0.276}*   &0.261   &0.135\\
$\widehat{\theta}_{2}^{ADG}$= 0.65   &S   &S  &S   &S   &NS  &NS  &LM      &9.2       &1.4   &0.303   &0.276   &\textbf{0.261}*   &0.129\\
$\widehat{\theta}_{2}^{WAV}$=  1.29   &S   &S  &S    &S    &NS  &NS &LM   &246.8      &10.1   &0.330   &0.227   &0.200   &\textbf{0.504}*\\
\hline
$\widehat{\theta}_{3}^{IR}$= 0.66   &S   &NS     &S    &S   &NS    &NS       &LM      &5.4    &1.1  &\textbf{0.444}*   &0.435 &0.461   &0.374\\
$\widehat{\theta}_{3}^{MS}$= 1.38   &S  &S    &S    &S  &NS    &NS   &LM     &64.8    &5.2  &0.402   &\textbf{0.492}*   &0.499   &0.391\\
$\widehat{\theta}_{3}^{ADG}$= 1.22   &S   &S  &S     &S  &NS    &NS     &LM    &36.2    &3.8  &0.400   &0.489  &\textbf{0.502}*   &0.387\\
$\widehat{\theta}_{3}^{WAV}$= 2.75   &S   &S   &S     &S  &NS   &NS      &LM  &1698.7   &35.8  &0.407   &0.315   &0.287   &\textbf{0.466}*\\
\hline \hline
\end{tabular}
~\\
\end{center}
}
\caption{{\small \label{Table12}Results of stationarity, nonstationarity and V/S tests and the $4$ long memory parameter estimators applied to several functionals $f$ of DowJones Transportation Index Return: from the top to bottom, $f(x)=x$, $f(x)=|x|$, $f(x)=x^2$ and $f(x)=|x|^\theta$ with $\theta$ maximizing the $4$ different long memory parameter estimators (''$S$`` for ''stationarity`` decision and ''$NS$`` for ''nonstationarity`` decision). Statistics are applied to the $3$ estimated stages of each trajectory (obtained from a change detection algorithm). }
}
\end{table}

\begin{table}[t13]
%{\tiny
%{\scriptsize
{\footnotesize
%{\small
\begin{center}
r=Dow Jones Utilities Index Return
\begin{tabular}{|c||c|c|c|c|c|c||c||c|c||c|c|c|c||}
\hline
\hline

Segments & \multicolumn{6}{c||}{(Non)Stationarity Test}& LRD & Kurtosis&Skewness &\multicolumn{4}{c||}{$\widehat {d}$}
\\
\cline{2-14}
Breaks&$\widetilde S_N$ & $\widetilde T_N$  & ADF & PP & KPSS & LMC &$V/S$ & $\widetilde{\kappa}$ & $\widetilde{s}$ &$\widetilde{d^{IR}}$ &$\widetilde{d^{MS}}$ &$\widehat{d^{ADG}}$ &$\widetilde{d^{WAV}}$
\\
\hline
$[1:1152]$ &S   &S  &S  &S   &S    &S  &SM  &7.3   &0.6    &0.191    &0.037    &-0.132    &0.222\\
\hline
$[1153:8748]$ &S   &S  &S   &S    &S    &S   &SM  &43.2   &-1.3    &0.094    &0.025    &0.001     &0.043 \\
\hline
$[8749:12071]$ &S    &S &S     &S    &S   &S    &SM   &13.0    &0.0    &0.026    &0.024    &0.001     &-0.032
\\
\hline \hline
\end{tabular}
~\\
\vspace{3mm}
$|r|=abs(\mbox{Dow Jones Utilities Index Return})$
\begin{tabular}{|c||c|c|c|c|c|c||c||c|c||c|c|c|c||}
\hline
\hline

Segments & \multicolumn{6}{c||}{(Non)Stationarity Test}& LRD & Kurtosis&Skewness &\multicolumn{4}{c||}{$\widehat {d}$}
\\
\cline{2-14}
Breaks&$\widetilde S_N$ & $\widetilde T_N$  & ADF & PP & KPSS & LMC &$V/S$ & $\widetilde{\kappa}$ & $\widetilde{s}$ &$\widetilde{d^{IR}}$ &$\widetilde{d^{MS}}$ &$\widehat{d^{ADG}}$ &$\widetilde{d^{WAV}}$
\\
\hline
$[1:1152]$ &S   &S   &S    &S   &NS    &NS    &LM     &11.9    &2.4    &0.283     &0.287   &0.316    &0.225\\
\hline
$[1153:8748]$ &S   &S  &S     &S    &NS    &NS    &LM    &127.4     &5.9    &0.134     &0.301   &0.304    &0.184\\
\hline
$[8749:12071]$ &S   &S &S    &S   &NS    &NS     &LM     &25.5     &3.4    &0.417     &0.559   &0.484    &0.595
\\
\hline \hline
\end{tabular}
~\\
\vspace{3mm}
$r^{2}=(\mbox{Dow Jones Utilities Index Return})^{2}$
\begin{tabular}{|c||c|c|c|c|c|c||c||c|c||c|c|c|c||}
\hline
\hline

Segments & \multicolumn{6}{c||}{(Non)Stationarity Test}& LRD & Kurtosis&Skewness &\multicolumn{4}{c||}{$\widehat {d}$}
\\
\cline{2-14}
Breaks&$\widetilde S_N$ & $\widetilde T_N$  & ADF & PP & KPSS & LMC &$V/S$ & $\widetilde{\kappa}$ & $\widetilde{s}$ &$\widetilde{d^{IR}}$ &$\widetilde{d^{MS}}$ &$\widehat{d^{ADG}}$ &$\widetilde{d^{WAV}}$
\\
\hline
$[1:1152]$ &S   &S  &S     &S     &NS     &S  &LM   &63.1      &6.7   &0.250    &0.253    &0.212   &0.270\\
\hline
$[1153:8748]$ &S   &S  &S     &S    &NS     &NS  &LM   &5322.4   &67.8   &0.130    &0.100    &0.100   &0.100 \\
\hline
$[8749:12071]$ &S  &NS    &S    &S     &NS     &NS    &LM   &289.6    &14.0   &0.510    &0.468    &0.423   &0.513
\\
\hline \hline
\end{tabular}
~\\
\vspace{3mm}
$|r|^{\theta}=(abs(\mbox{Dow Jones Utilities Index Return}))^{\theta}$
\begin{tabular}{|c||c|c|c|c|c|c||c||c|c||c|c|c|c||}
\hline
\hline

$\widehat{\theta}_{i}^{(j)}=$ & \multicolumn{6}{c||}{(Non)Stationarity Test}& LRD& Kurtosis&Skewness &\multicolumn{4}{c||}{$\widehat{d}$}
\\
\cline{2-14}
$\mbox{ArgMax}_\theta(\widehat{d}(|r_{i}|^{\theta}))$&$\widetilde S_N$ & $\widetilde T_N$   & ADF & PP & KPSS & LMC &$V/S$ & $\widetilde{\kappa}$ & $\widetilde{s}$ &$\widetilde{d^{IR}}$ &$\widetilde{d^{MS}}$ &$\widehat{d^{ADG}}$ &$\widetilde{d^{WAV}}$
\\
\hline
$\widehat{\theta}_{1}^{IR}= 0.39 $   &S    &S    &NS    &S      &NS    &NS  &LM    &3.3   &-0.1    &\textbf{0.354}*   &0.262   &0.327    &0.145\\
$\widehat{\theta}_{1}^{MS}= 1.09  $  &S    &S    &S   &S     &NS   &NS &LM   &14.6    &8.4    &0.215   &\textbf{0.288}*    &0.308    &0.234\\
$\widehat{\theta}_{1}^{ADG}= 0.60  $  &S    &S   &NS   &S      &NS    &NS  &LM    &4.5    &0.8    &0.311   &0.276    &\textbf{0.336}*    &0.396\\
$\widehat{\theta}_{1}^{WAV}= 0.63   $ &S    &S   &S   &S    &NS    &NS  &LM    &4.8    &1.0    &0.310   &0.278    &0.336    &\textbf{0.398}*\\
\hline
$\widehat{\theta}_{2}^{IR}= 3.00  $ &S  &S &S   &S     &S  &S  &LM   &7320.6   &84.9   &\textbf{0.165}*   &0.015   &0.017   &0.040\\
$\widehat{\theta}_{2}^{MS}= 0.61  $ &S  &S &S    &S    &NS    &NS   &LM   &9.1       &1.2   &0.113   &\textbf{0.330}*   &0.327   &0.113\\
$\widehat{\theta}_{2}^{ADG}= 0.67  $ &S   &S &S    &S   &NS    &NS  &LM   &13.0      &1.6  &0.117   &0.330   &\textbf{0.327}*   &0.113\\
$\widehat{\theta}_{2}^{WAV}= 1.84  $ &S   &S &S    &S   &NS    &NS    &LM  &4386.6   &59.1   &0.125   &0.130   &0.129   &\textbf{0.377}*\\
\hline
$\widehat{\theta}_{3}^{IR}= 2.69   $ &S    &NS   &S  &S   &NS        &NS  &LM  &683.1   &22.5   &\textbf{0.527}*   &0.394    &0.344  &0.426\\
$\widehat{\theta}_{3}^{MS}= 0.95   $ &S   &S   &S    &S    &NS       &NS  &LM   &21.6    &3.1    &0.415   &\textbf{0.560}*    &0.483   &0.544\\
$\widehat{\theta}_{3}^{ADG}= 1.10  $  &S  &S  &S    &S  &NS  &NS  &LM   &35.2    &4.2    &0.421   &0.557    &\textbf{0.485}*    &0.364\\
$\widehat{\theta}_{3}^{WAV}= 1.03  $  &S  &S  &S      &S  &NS   &NS &LM  &28.2    &3.7    &0.419  &0.559    &0.484   &\textbf{0.723}*\\
\hline \hline
\end{tabular}
~\\
\end{center}
}
\caption{{\small \label{Table13} Results of stationarity, nonstationarity and V/S tests and the $4$ long memory parameter estimators applied to several functionals $f$ of Dow Jones Utilities Index Return: from the top to bottom, $f(x)=x$, $f(x)=|x|$, $f(x)=x^2$ and $f(x)=|x|^\theta$ with $\theta$ maximizing the $4$ different long memory parameter estimators (''$S$`` for ''stationarity`` decision and ''$NS$`` for ''nonstationarity`` decision). Statistics are applied to the $3$ estimated stages of each trajectory (obtained from a change detection algorithm). }
}
\end{table}

\begin{table}[t14]

{\footnotesize

\begin{center}
r=Nasdaq Industrials Index Return
\begin{tabular}{|c||c|c|c|c|c|c||c||c|c||c|c|c|c||}
\hline
\hline

Segments & \multicolumn{6}{c||}{(Non)Stationarity Test}& LRD & Kurtosis&Skewness &\multicolumn{4}{c||}{$\widehat {d}$}
\\
\cline{2-14}
Breaks&$\widetilde S_N$ & $\widetilde T_N$  & ADF & PP & KPSS & LMC &$V/S$ & $\widetilde{\kappa}$ & $\widetilde{s}$ &$\widetilde{d^{IR}}$ &$\widetilde{d^{MS}}$ &$\widehat{d^{ADG}}$ &$\widetilde{d^{WAV}}$
\\
\hline
$[1:7160]$ &S   &S  &S    &S  &S  &S &SM  &20.7   &-1.5    &0.141    &0.073    &0.092   &-0.202\\
\hline
$[7161:8320]$ &S   &S  &S  &S  &S  &S  &SM  &4.6   &0.0    &0.012    &0.070    &0.116    &0.014\\
\hline
$[8321:10480]$ &S  &S  &S  &S   &NS  &S &SM   &10.4   &-0.3    &0.045    &0.078    &0.082   &-0.045\\

\hline \hline
\end{tabular}
~\\
\vspace{3mm}
$|r|=abs(\mbox{Nasdaq Industrials Index Return})$
\begin{tabular}{|c||c|c|c|c|c|c||c||c|c||c|c|c|c||}
\hline
\hline

Segments & \multicolumn{6}{c||}{(Non)Stationarity Test}& LRD & Kurtosis&Skewness &\multicolumn{4}{c||}{$\widehat {d}$}
\\
\cline{2-14}
Breaks&$\widetilde S_N$ & $\widetilde T_N$  & ADF & PP & KPSS & LMC &$V/S$ & $\widetilde{\kappa}$ & $\widetilde{s}$ &$\widetilde{d^{IR}}$ &$\widetilde{d^{MS}}$ &$\widehat{d^{ADG}}$ &$\widetilde{d^{WAV}}$
\\
\hline
$[1:7160]$ &S  &S &S   &S &S  &NS  &SM  &52.4    &4.4    &0.361    &0.309   &0.287    &0.274\\
\hline
$[7161:8320]$ &S   &S  &S &S  &NS  &NS   &LM       &7.4     &1.6    &0.284    &0.532   &0.504    &0.385\\
\hline
$[8321:10480]$ &S   &NS  &S  &NS   &NS  &NS   &LM     &18.3    &3.0    &0.516    &0.761   &0.606    &0.668\\
\hline \hline
\end{tabular}
~\\
\vspace{3mm}
$r^{2}=(\mbox{Nasdaq Industrials Index Return})^{2}$
\begin{tabular}{|c||c|c|c|c|c|c||c||c|c||c|c|c|c||}
\hline
\hline

Segments & \multicolumn{6}{c||}{(Non)Stationarity Test}& LRD & Kurtosis&Skewness &\multicolumn{4}{c||}{$\widehat {d}$}
\\
\cline{2-14}
Breaks&$\widetilde S_N$ & $\widetilde T_N$  & ADF & PP & KPSS & LMC &$V/S$ & $\widetilde{\kappa}$ & $\widetilde{s}$ &$\widetilde{d^{IR}}$ &$\widetilde{d^{MS}}$ &$\widehat{d^{ADG}}$ &$\widetilde{d^{WAV}}$
\\
\hline
$[1:7160]$ &S  &S &S   &S  &S  &NS  &SM   &1356.8   &31.4   &0.381   &0.146   &0.114   &0.100\\
\hline
$[7161:8320]$ &S  &S &S   &S &NS  &NS &LM    &49.0   &5.4   &0.304    &0.466   &0.378   &0.432\\
\hline
$[8321:10480]$ &S &NS &S    &S    &NS  &NS &LM   &140.0   &10.0   &0.498    &0.786   &0.544   &0.708\\
\hline \hline
\end{tabular}
~\\
\vspace{3mm}
$|r|^{\theta}=(abs(\mbox{Nasdaq Industrials Index Return}))^{\theta}$
\begin{tabular}{|c||c|c|c|c|c|c||c||c|c||c|c|c|c||}
\hline
\hline

$\widehat{\theta}_{i}^{(j)}=$ & \multicolumn{6}{c||}{(Non)Stationarity Test}& LRD& Kurtosis&Skewness &\multicolumn{4}{c||}{$\widehat{d}$}
\\
\cline{2-14}
$\mbox{ArgMax}_\theta(\widehat{d}(|r_{i}|^{\theta}))$&$\widetilde S_N$ & $\widetilde T_N$   & ADF & PP & KPSS & LMC &$V/S$ & $\widetilde{\kappa}$ & $\widetilde{s}$ &$\widetilde{d^{IR}}$ &$\widetilde{d^{MS}}$ &$\widehat{d^{ADG}}$ &$\widetilde{d^{WAV}}$
\\
\hline
$\widehat{\theta}_{1}^{IR}= 1.04 $   &S   &S  &S  &S  &S   &NS   &LM   &63.7    &4.9    &\textbf{0.396}*   &0.304    &0.281    &0.325\\
$\widehat{\theta}_{1}^{MS}= 0.67  $  &S   &S  &S  &S    &NS     &NS &LM   &10.2    &0.7    &0.188   &\textbf{0.329}*    &0.325    &0.293\\
$\widehat{\theta}_{1}^{ADG}=  0.56 $   &S   &S &S  &S    &NS    &NS  &LM   &6.3    &0.9    &0.178   &0.326  &\textbf{0.328}*    &0.275\\
$\widehat{\theta}_{1}^{WAV}=  0.83  $  &S   &S &S  &S  &NS   &NS &LM  &22.4    &2.6   &0.199    &0.324    &0.311  &\textbf{0.587}*\\
\hline
$\widehat{\theta}_{2}^{IR}= 2.83  $  &S   &S &S &S &NS    &S &LM  &142.0    &9.8    &\textbf{0.317}*   &0.374    &0.276   &0.263\\
$\widehat{\theta}_{2}^{MS}= 1.03  $  &S    &S &S  &S   &NS  &NS  &LM   &7.9    &1.7    &0.284   &\textbf{0.532}*    &0.501    &0.388\\
$\widehat{\theta}_{2}^{ADG}= 0.73  $  &S  &S &NS   &S    &NS   &NS &LM    &4.2    &0.9    &0.300   &0.517    &\textbf{0.517}*    &0.340\\
$\widehat{\theta}_{2}^{WAV}= 1.87   $ &S  &S &S   &S   &NS   &NS  &LM   &39.8   &4.8    &0.299   &0.479   &0.395    &\textbf{0.432}*\\
\hline
$\widehat{\theta}_{3}^{IR}= 2.60   $ &S  &NS  &S     &S &NS   &S &LM  &256.9   &14.3    &\textbf{0.548}*   &0.272    &0.479    &0.669\\
$\widehat{\theta}_{3}^{MS}= 1.70   $ &S &NS &S     &S  &NS &NS &LM  &89.1    &7.7    &0.504   &\textbf{0.801}*    &0.575    &0.739\\
$\widehat{\theta}_{3}^{ADG}= 1.13  $  &S  &NS &S  &S  &NS &NS &LM  &26.0    &3.7    &0.526   &0.772   &\textbf{0.608}*    &0.671\\
$\widehat{\theta}_{3}^{WAV}= 1.26  $  &S    &NS  &S    &S    &NS   &NS &LM  &36.0    &4.5    &0.532   &0.782   &0.606  &\textbf{0.760}*\\
\hline \hline
\end{tabular}
~\\
\end{center}
}
\caption{{\small \label{Table14} Results of stationarity, nonstationarity and V/S tests and the $4$ long memory parameter estimators applied to several functionals $f$ of Nasdaq Industrials Index Return: from the top to bottom, $f(x)=x$, $f(x)=|x|$, $f(x)=x^2$ and $f(x)=|x|^\theta$ with $\theta$ maximizing the $4$ different long memory parameter estimators (''$S$`` for ''stationarity`` decision and ''$NS$`` for ''nonstationarity`` decision). Statistics are applied to the $3$ estimated stages of each trajectory (obtained from a change detection algorithm). }
}
\end{table}

\begin{table}[t15]

{\footnotesize

\begin{center}
r=Nikkei 225A Index Return
\begin{tabular}{|c||c|c|c|c|c|c||c||c|c||c|c|c|c||}
\hline
\hline

Segments & \multicolumn{6}{c||}{(Non)Stationarity Test}& LRD & Kurtosis&Skewness &\multicolumn{4}{c||}{$\widehat {d}$}
\\
\cline{2-14}
Breaks&$\widetilde S_N$ & $\widetilde T_N$  & ADF & PP & KPSS & LMC &$V/S$ & $\widetilde{\kappa}$ & $\widetilde{s}$ &$\widetilde{d^{IR}}$ &$\widetilde{d^{MS}}$ &$\widehat{d^{ADG}}$ &$\widetilde{d^{WAV}}$
\\
\hline
$[1:6672]$ &S   &S   &S   &S   &S   &S &SM   &12.6   &-0.6    &0.083    &0.067    &0.084    &0.022\\
\hline
$[6673:10400]$ &S &S  &S   &S  &S &S &SM   &63.7   &-2.3  &-0.021   &-0.016    &-0.013   &-0.039\\
\hline
$[10401:15919]$ &S &S  &S  &S  &S   &S  &SM  &9.0   &-0.1    &0.033    &0.047   &-0.005   &-0.015
\\
\hline \hline
\end{tabular}
~\\
\vspace{3mm}
$|r|=abs(\mbox{Nikkei 225A Index Return})$
\begin{tabular}{|c||c|c|c|c|c|c||c||c|c||c|c|c|c||}
\hline
\hline

Segments & \multicolumn{6}{c||}{(Non)Stationarity Test}& LRD & Kurtosis&Skewness &\multicolumn{4}{c||}{$\widehat {d}$}
\\
\cline{2-14}
Breaks&$\widetilde S_N$ & $\widetilde T_N$  & ADF & PP & KPSS & LMC &$V/S$ & $\widetilde{\kappa}$ & $\widetilde{s}$ &$\widetilde{d^{IR}}$ &$\widetilde{d^{MS}}$ &$\widehat{d^{ADG}}$ &$\widetilde{d^{WAV}}$
\\
\hline
$[1:6672]$ &S    &S &S &S  &NS    &NS   &LM   &26.1     &3.3     &0.302    &0.343   &0.313    &0.218\\
\hline
$[6673:10400]$ &S  &S &S   &S  &NS  &NS  &LM &150.9    &7.5    &0.196    &0.346   &0.304    &0.321\\
\hline
$[10401:15919]$ &S   &S  &S  &S  &NS  &NS   &LM  &17.0     &2.6     &0.413    &0.415   &0.431    &0.335
\\
\hline \hline
\end{tabular}
~\\
\vspace{3mm}
$r^{2}=(\mbox{Nikkei 225A Index Return})^{2}$
\begin{tabular}{|c||c|c|c|c|c|c||c||c|c||c|c|c|c||}
\hline
\hline

Segments & \multicolumn{6}{c||}{(Non)Stationarity Test}& LRD & Kurtosis&Skewness &\multicolumn{4}{c||}{$\widehat {d}$}
\\
\cline{2-14}
Breaks&$\widetilde S_N$ & $\widetilde T_N$  & ADF & PP & KPSS & LMC &$V/S$ & $\widetilde{\kappa}$ & $\widetilde{s}$ &$\widetilde{d^{IR}}$ &$\widetilde{d^{MS}}$ &$\widehat{d^{ADG}}$ &$\widetilde{d^{WAV}}$
\\
\hline
$[1:6672]$ &S  &S &S   &S  &NS   &NS &LM   &427.8  &16.9  &0.275   &0.241   &0.267   &0.080\\
\hline
$[6673:10400]$ &S  &S &S   &S   &S   &NS &LM  &2610.8 &48.0 &0.230   &0.146   &0.154   &0.117\\
\hline
$[10401:15919]$ &S  &S  &S   &S  &NS   &NS &LM  &235.5   &12.3   &0.381   &0.396   &0.363   &0.377
\\
\hline \hline
\end{tabular}
~\\
\vspace{3mm}
$|r|^{\theta}=(abs(\mbox{Nikkei 225A Index Return}))^{\theta}$
\begin{tabular}{|c||c|c|c|c|c|c||c||c|c||c|c|c|c||}
\hline
\hline

$\widehat{\theta}_{i}^{(j)}=$ & \multicolumn{6}{c||}{(Non)Stationarity Test}& LRD& Kurtosis&Skewness &\multicolumn{4}{c||}{$\widehat{d}$}
\\
\cline{2-14}
$\mbox{ArgMax}_\theta(\widehat{d}(|r_{i}|^{\theta}))$&$\widetilde S_N$ & $\widetilde T_N$   & ADF & PP & KPSS & LMC &$V/S$ & $\widetilde{\kappa}$ & $\widetilde{s}$ &$\widetilde{d^{IR}}$ &$\widetilde{d^{MS}}$ &$\widehat{d^{ADG}}$ &$\widetilde{d^{WAV}}$
\\
\hline
$\widehat{\theta}_{1}^{IR}= 1.53 $   &S  &S  &S   &S  &NS  &NS  &LM  &149.8    &9.0    &\textbf{0.323}*   &0.296    &0.296    &0.201\\
$\widehat{\theta}_{1}^{MS}= 0.86  $  &S  &S  &S  &S  &NS    &NS  &LM   &15.4    &4.4    &0.286   &\textbf{0.345}*    &0.311    &0.213\\
$\widehat{\theta}_{1}^{ADG}= 1.01  $  &S   &S &S  &S  &NS  &NS &LM   &27.1    &3.4    &0.303   &0.342    &\textbf{0.313}*    &0.218\\
$\widehat{\theta}_{1}^{WAV}= 1.30 $   &S   &S  &S  &S  &NS  &NS &LM   &74.8    &6.1    &0.273   &0.322    &0.307    &\textbf{0.622}*\\
\hline
$\widehat{\theta}_{2}^{IR}= 3.00 $  &S  &S &S &S   &S  &S  &LM   &3487.4     &58.3  &\textbf{0.252}*   &0.037   &0.042   &0.045\\
$\widehat{\theta}_{2}^{MS}= 0.84 $  &S  &S  &S    &S  &NS   &NS  &SM    &55.5      &4.1   &0.180   &\textbf{0.353}*   &0.304  &0.154\\
$\widehat{\theta}_{2}^{ADG}= 0.91 $  &S  &S &S  &S  &NS   &NS  &SM     &87.1      &5.3   &0.186   &0.352   &\textbf{0.305}*   &0.035\\
$\widehat{\theta}_{2}^{WAV}= 1.64  $ &S  &S &S   &S &S  &NS  &SM   &1697.1     &35.8   &0.221   &0.220   &0.222   &\textbf{0.465}*\\
\hline
$\widehat{\theta}_{3}^{IR}= 1.23  $  &S   &NS  &S  &S  &NS  &NS  &LM   &35.7  &4.1    &\textbf{0.467}*  &0.412    &0.426    &0.386\\
$\widehat{\theta}_{3}^{MS}= 0.87  $  &S   &S  &S   &S  &NS  &NS   &LM &11.0    &2.0    &0.429  &\textbf{0.415}*    &0.428    &0.351\\
$\widehat{\theta}_{3}^{ADG}= 1.00  $  &S   &S &S   &S &NS  &NS   &LM   &17.0   &2.6   &0.413  &0.415    &\textbf{0.431}*    &0.335\\
$\widehat{\theta}_{3}^{WAV}= 1.27  $  &S   &NS &S   &S   &NS  &NS   &LM   &40.3    &4.4    &0.467  &0.411    &0.424    &\textbf{0.425}*\\
\hline \hline
\end{tabular}
~\\
\end{center}
}
\caption{{\small \label{Table15} Results of stationarity, nonstationarity and V/S tests and the $4$ long memory parameter estimators applied to several functionals $f$ of Nikkei 225A Index Return: from the top to bottom, $f(x)=x$, $f(x)=|x|$, $f(x)=x^2$ and $f(x)=|x|^\theta$ with $\theta$ maximizing the $4$ different long memory parameter estimators (''$S$`` for ''stationarity`` decision and ''$NS$`` for ''nonstationarity`` decision). Statistics are applied to the $3$ estimated stages of each trajectory (obtained from a change detection algorithm). }
}
\end{table}

\section{Proofs}\label{proofs} \label{Appendix}

\begin{proof}[Proof of Proposition \ref{MCLT}]
This proposition is based on results of Surgailis {\it et al.} (2008) and was already proved in Bardet et Dola (2012) in the case $-0.5<d<0.5$. \\
{\it Mutatis mutandis}, the case $0.5<d<1.25$ can be treated exactly following the same steps. \\
The only new proof which has to be established concerns the case $d=0.5$ since Surgailis {\it et al.} (2008) do not provide a CLT satisfied by the (unidimensional) statistic $IR_N(m)$ in this case. Let $Y_{m}(j)$ the standardized process defined Surgailis {\it et al.} (2008). Then, for $d=0.5$,  \begin{eqnarray*}\label{Rmj}
\forall j \geq 1,\quad |\gamma_{m}(j)|=\big |\E \big ( Y_{m}(j) Y_{m}(0) \big ) \big | = \frac{2}{V^{2}_{m}} \Big |\int_{0}^{\pi}~\cos(jx)~x\big(c_{0}+O(x^{\beta})\big)\frac{\sin^{4}(\frac{mx}{2})}{\sin^{4}(\frac{x}{2})}dx \Big |.
\end{eqnarray*}
Denote $\gamma_{m}(j)=\rho_{m}(j)=\frac{2}{V^{2}_{m}}\big(I_{1}+I_{2}\big)$ as in (5.39) of Surgailis {\it et al.} (2008). Both inequalities (5.41) and (5.42) remain true for $d=0.5$ and
\begin{multline*}
|I_{1}|\leq C\, \frac{m^{3}}{j}, \quad |I_{2}|\leq C\, \frac{m^{4}}{j^{2}}\quad  \Longrightarrow \quad |I_{1}+I_{2}|\leq C\, \frac{m^{3}}{j} \quad
\Longrightarrow \quad  |\gamma_{m}(j)|=|\rho_{m}(j)|\leq\frac{2}{V^{2}_{m}}\big(|I_{1}+I_{2}|\big)\leq C\, \frac{m}{j}.
\end{multline*}
Now let $\eta_{m}(j):=\frac{|Y_{m}(j)+Y_{m}(j+m)|}{|Y_{m}(j)|+|Y_{m}(j+m)|}:=\psi\Big(Y_{m}(j),Y_{m}(j+m)\Big)$. The Hermite rank of the function $\psi$ is $2$ and therefore the equation (5.23) of Surgailis {\it et al.} (2008) obtained from Arcones Lemma remains valid. Hence:
$$\big |\Cov(\eta_{m}(0),\eta_{m}(j))\big |\leq C\frac{m^{2}}{j^{2}}$$
from Lemma (8.2) and then the equations (5.28-5.31) remain valid for all $d\in[0.5,1.25)$. Then for $d=0.5$,
\begin{eqnarray*}
\sqrt{\frac{N}{m}}\, \Big(IR_N(m)-\E \big [ IR_N(m)\big ]\Big)\limiteloiNm \mathcal{N}\big(0,\sigma^{2}(0.5)\big),
\end{eqnarray*}
with $\sigma^{2}(0.5)\simeq (0.2524)^{2}$.
\end{proof}

\begin{proof}[Proof of Property \ref{devEIR}]
As in Surgailis {\it et al} (2008), we can write:
\begin{eqnarray} \label{RmVm}
\nonumber \E \big [ IR_N(m) \big ]=\E\big(\frac{|Y^{0}+Y^{1}|}{|Y^{0}|+|Y^{1}|}\big)=\Lambda(\frac{R_{m}}{V_{m}^{2}}) \quad
\mbox{with}\quad
\frac{R_{m}}{V_{m}^{2}}:=1-2 \, \frac{\int_{0}^{\pi}f(x)\frac{\sin^{6}(\frac{mx}{2})}{\sin^{2}(\frac{x}{2})}dx~}{\int_{0}^{\pi}f(x)
\frac{\sin^{4}(\frac{mx}{2})}{\sin^{2}(\frac{x}{2})}dx}.
\end{eqnarray}
Therefore an expansion of $R_{m}/V_{m}^{2}$ provides an expansion of $\E \big [ IR_N(m) \big ]$ when $m\to \infty$.\\
~\\
{\bf Step 1} Let $f$ satisfy Assumption $IG(d,\beta)$. Then we are going to establish that there exist positive real numbers ${C}_{1}$, ${C}_{2}$ and ${C}_{3}$ specified in (\ref{Ctild1}), (\ref{Ctild2}) and (\ref{Ctild3}) such that for $0.5\leq d<1.5$ and with $\rho(d)$ defined in (\ref{DefinitionRhod}),
\begin{eqnarray*}
&1.& \quad \mbox{if $ \beta < 2d-1$,} \quad  \frac{R_{m}}{V_{m}^{2}}=\rho(d)+{C}_{1}(2-2d,\beta)m^{-\beta}+O\Big(m^{-2} +m^{-2\beta}\Big); \\
&2.&\quad \mbox{if $ \beta = 2d-1$,} \quad  \frac{R_{m}}{V_{m}^{2}}=\rho(d)+{C}_{2}(2-2d,\beta)m^{-\beta}+O\Big(m^{-2}+m^{-2-\beta}\log(m)+m^{-2\beta} \Big); \\
&3.&\quad \mbox{if $ 2d-1<\beta <2d+1$,} \quad  \frac{R_{m}}{V_{m}^{2}}=\rho(d)+{C}_{3}(2-2d,\beta)m^{-\beta}+O\Big(m^{-\beta-\epsilon}+m^{-2d-1}\log(m)+m^{-2\beta}\Big);\\
&4.&\quad\mbox{if $\beta =2d+1$,} \quad  \frac{R_{m}}{V_{m}^{2}}=\rho(d)+O\Big(m^{-2d-1}~\log(m)+m^{-2} \Big).
\end{eqnarray*}
Under Assumption $IG(d,\beta)$ and with $J_j(a,m)$ defined in (\ref{Jj}) in Lemma \ref{Jj}, it is clear that,
$$
\frac{R_{m}}{V_{m}^{2}}=1-2\, \frac{J_6(2-2d,m)+ \frac{c_{1}}{c_{0}}J_6(2-2d+\beta,m) +O(J_6(2-2d+\beta+\varepsilon))}
{J_4(2-2d,m)+ \frac{c_{1}}{c_{0}}J_4(2-2d+\beta,m) +O(J_4(2-2d+\beta+\varepsilon))},
$$
since $\displaystyle \int_{0}^{\pi}O(x^{2-2d+\beta+\varepsilon})\frac{\sin^{j}(\frac{mx}{2})}{\sin^{2}(\frac{x}{2})}dx=O(J_j(2-2d+\beta+\varepsilon))$. Now using the results of Lemma \ref{Jj} and constants ${C}_{j\ell}$, ${C}'_{j\ell}$ and ${C}''_{j\ell}$, $j=4,\, 6$, $\ell=1,2$ defined in Lemma \ref{Jj},\\
~\\
1. Let $0<\beta<2d-1<2$, {\it i.e.} $-1<2-2d+\beta<1$. Then
\begin{eqnarray*}
\frac{R_{m}}{V_{m}^{2}}& \hspace{-3mm}=&\hspace{-3mm} 1\hspace{-1mm} -\hspace{-1mm}2\, \frac{
{C}_{61}(2-2d)~m^{1+2d}\hspace{-1mm}+\hspace{-1mm}O\big(m^{2d-1}\big)+\hspace{-1mm}\frac{c_{1}}{c_{0}}{C}_{61}(2-2d+\beta)m^{1+2d-\beta}\hspace{-1mm}+\hspace{-1mm}O\big(m^{2d-1-\beta}\big)}
{{C}_{41}(2-2d)m^{1+2d}\hspace{-1mm}+\hspace{-1mm}O\big(m^{2d-1}\big)+\hspace{-1mm}\frac{c_{1}}{c_{0}} {C}_{41}(2-2d+\beta)m^{1+2d-\beta}\hspace{-1mm}+\hspace{-1mm}O\big(m^{2d-1-\beta}\big)}\\
&\hspace{-3mm}=&\hspace{-3mm}1\hspace{-1mm}-\hspace{-1mm}\frac{2}{{C}_{41}(2-2d)}\Big[{C}_{61}(2-2d)\hspace{-1mm}+\hspace{-1mm}\frac{c_{1}}{c_{0}}{C}_{61}(2-2d+\beta)m^{-\beta}\Big]\Big[1\hspace{-1mm}-\hspace{-1mm}\frac{c_{1}}{c_{0}}\frac{{C}_{41}(2-2d+\beta)}{{C}_{41}(2-2d)}m^{-\beta}\Big]\hspace{-1mm}+\hspace{-1mm}O\big(m^{-2}\big)\\
&\hspace{-3mm}=&\hspace{-3mm}1\hspace{-1mm}-\hspace{-1mm}\frac{2{C}_{61}(2-2d)}{{C}_{41}(2-2d)}\hspace{-1mm}+\hspace{-1mm}2\frac{c_{1}}{c_{0}}\Big[\frac{{C}_{61}(2-2d){C}_{41}(2-2d+\beta)}{{C}_{41}(2-2d){C}_{41}(2-2d)}\hspace{-1mm}-\hspace{-1mm}\frac{{C}_{61}(2-2d+\beta)}{{C}_{41}(2-2d)}\Big]m^{-\beta}\hspace{-1mm}+\hspace{-1mm}O\big(m^{-2}+m^{-2\beta}\big).
\end{eqnarray*}
As a consequence,,
\begin{multline}\label{Ctild1}
\frac{R_{m}}{V_{m}^{2}}=\rho(d)~+~{C}_{1}(2-2d,\beta)~~m^{-\beta}+~O\Big(m^{-2}+m^{-2\beta} \Big)\quad (m\to \infty),\quad \mbox{with $0<\beta<2d-1<2$ and} \\
{C}_{1}(2-2d,\beta):=2 \, \frac{c_{1}}{c_{0}}\frac 1 {{C}^2_{41}(2-2d)}\big[{C}_{61}(2-2d){C}_{41}(2-2d+\beta)-{C}_{61}(2-2d+\beta)
{C}_{41}(2-2d)\big],
\end{multline}
and numerical experiments proves that $ {C}_{1}(2-2d,\beta)/c_1$ is negative for any $d \in (0.5,1.5)$ and $\beta>0$. \\
~\\
2. Let $\beta=2d-1$, {\it i.e.} $2-2d+\beta=1$. Then,
\begin{eqnarray*}
\frac{R_{m}}{V_{m}^{2}}& \hspace{-3mm}=&\hspace{-3mm} 1\hspace{-1mm} -\hspace{-1mm}2\, \frac{
{C}_{61}(2-2d)~m^{1+2d}\hspace{-1mm}+\hspace{-1mm}O\big(m^{2d-1}\big)+\hspace{-1mm}\frac{c_{1}}{c_{0}}{C'}_{61}(1)m^{1-2d}\hspace{-1mm}+\hspace{-1mm}O\big(\log(m)\big)}
{{C}_{41}(2-2d)m^{1+2d}\hspace{-1mm}+\hspace{-1mm}O\big(m^{2d-1}\big)+\hspace{-1mm}\frac{c_{1}}{c_{0}} {C'}_{41}(1)m^{1-2d}\hspace{-1mm}+\hspace{-1mm}O\big(\log(m)\big)}\\
&\hspace{-3mm}=&\hspace{-3mm}1\hspace{-1mm}-\hspace{-1mm}\frac{2}{{C}_{41}(2-2d)}\Big[{C}_{61}(2-2d)\hspace{-1mm}+\hspace{-1mm}\frac{c_{1}}{c_{0}}{C'}_{61}(1)m^{1-2d}\Big]\Big[1\hspace{-1mm}-\hspace{-1mm}\frac{c_{1}}{c_{0}}\frac{{C'}_{41}(1)}{{C}_{41}(2-2d)}m^{1-2d}\Big]\hspace{-1mm}+\hspace{-1mm}O\big(m^{-2}+m^{-2d-1}\log(m)\big)\\
&\hspace{-3mm}=&\hspace{-3mm}1\hspace{-1mm}-\hspace{-1mm}\frac{2{C}_{61}(2-2d)}{{C}_{41}(2-2d)}\hspace{-1mm}+\hspace{-1mm}2\frac{c_{1}}{c_{0}}\Big[\frac{{C}_{61}(2-2d){C'}_{41}(1)}{{C}_{41}(2-2d){C}_{41}(2-2d)}\hspace{-1mm}-\hspace{-1mm}\frac{{C'}_{61}(1)}{{C}_{41}(2-2d)}\Big]m^{1-2d}\hspace{-1mm}+\hspace{-1mm}O\big(m^{-2}+m^{-2d-1}\log(m)+m^{2-4d}\big).
\end{eqnarray*}
As a consequence,
\begin{multline}\label{Ctild2}
\frac{R_{m}}{V_{m}^{2}}=\rho(d)~+~{C}_{2}(2-2d,\beta)~~m^{-\beta}+~O\Big(m^{-2}+m^{-2-\beta}\log(m)+m^{-2\beta} \Big)\quad (m\to \infty),\quad \mbox{with $0<\beta=2d-1<2$ and } \\
{C}_{2}(2-2d,\beta):=2 \, \frac{c_{1}}{c_{0}}\frac 1 {{C}^2_{41}(2-2d)}\big[{C}_{61}(2-2d){C'}_{41}(1)-{C'}_{61}(1)
{C}_{41}(2-2d)\big],
\end{multline}
and numerical experiments proves that $ {C}_{2}(2-2d,\beta)/c_1$ is negative for any $d \in [0.5,1.5)$ and $\beta>0$. \\\\
3.  Let $2d-1<\beta<2d+1$,  {\it i.e.} $1<2-2d+\beta<3$. Then,
\begin{eqnarray*}
\frac{R_{m}}{V_{m}^{2}}& \hspace{-3mm}=&\hspace{-3mm} 1\hspace{-1mm} -\hspace{-1mm}2\, \frac{
{C}_{61}(2-2d)m^{1+2d}\hspace{-1mm}+\hspace{-1mm}\frac{c_{1}}{c_{0}}{C'}_{61}(2-2d+\beta)m^{1+2d-\beta}\hspace{-1mm}+\hspace{-1mm}O\big(m^{1+2d-\beta-\epsilon}+\log(m)\big)}
{{C}_{41}(2-2d)m^{1+2d}\hspace{-1mm}+\hspace{-1mm}\frac{c_{1}}{c_{0}} {C'}_{41}(2-2d+\beta)m^{1+2d-\beta}\hspace{-1mm}+\hspace{-1mm}O\big(m^{1+2d-\beta-\epsilon}+m^{-2d-1}\log(m)\big)}\\
&\hspace{-3mm}=&\hspace{-3mm}1\hspace{-1mm}-\hspace{-1mm}\frac{2}{{C}_{41}(2-2d)}\Big[{C}_{61}(2-2d)\hspace{-1mm}+\hspace{-1mm}\frac{c_{1}}{c_{0}}{C'}_{61}(2-2d+\beta)m^{-\beta}\Big]\Big[1\hspace{-1mm}-\hspace{-1mm}\frac{c_{1}}{c_{0}}\frac{{C'}_{41}(2-2d+\beta)}{{C}_{41}(2-2d)}m^{-\beta}\Big]\hspace{-1mm}+\hspace{-1mm}O\big(m^{-\beta-\epsilon}+m^{-2d-1}\log(m)\big)\\
&\hspace{-3mm}=&\hspace{-3mm}1\hspace{-1mm}-\hspace{-1mm}\frac{2{C}_{61}(2-2d)}{{C}_{41}(2-2d)}\hspace{-1mm}+\hspace{-1mm}2\frac{c_{1}}{c_{0}}\Big[\frac{{C}_{61}(2-2d){C'}_{41}(2-2d+\beta)}{{C}_{41}(2-2d){C}_{41}(2-2d)}\hspace{-1mm}-\hspace{-1mm}\frac{{C'}_{61}(2-2d+\beta)}{{C}_{41}(2-2d)}\Big]m^{-\beta}\hspace{-1mm}+\hspace{-1mm}O\big(m^{-\beta-\epsilon}+m^{-2d-1}\log(m)\big).
\end{eqnarray*}
As a consequence,
\begin{multline}\label{Ctild3}
\frac{R_{m}}{V_{m}^{2}}=\rho(d)~+~{C}_{3}(2-2d,\beta)~~m^{-\beta}+~O\Big(m^{-\beta-\epsilon}+m^{-2d-1}\log(m)+m^{-2\beta} \Big)\quad (m\to \infty),\quad \mbox{and} \\
{C}_{3}(2-2d,\beta):=2 \, \frac{c_{1}}{c_{0}}\frac 1 {{C}^2_{41}(2-2d)}\big[{C}_{61}(2-2d){C'}_{41}(2-2d+\beta)-{C'}_{61}(2-2d+\beta)
{C}_{41}(2-2d)\big],
\end{multline}
and numerical experiments proves that $ {C}_{3}(2-2d,\beta)/c_1$ is negative for any $d \in [0.5,1.5)$ and $\beta>0$. \\
~\\
4. Let $\beta=2d+1$. Then,
Once again with Lemma \ref{Jj}:
\begin{eqnarray*}
\frac{R_{m}}{V_{m}^{2}}& \hspace{-3mm}=&\hspace{-3mm} 1\hspace{-1mm} -\hspace{-1mm}2\, \frac{
{C}_{61}(2-2d)~m^{1+2d}\hspace{-1mm}+\hspace{-1mm}O\big(m^{2d-1}\big)+\hspace{-1mm}\frac{c_{1}}{c_{0}}{C'}_{62}(3)\log(m)\hspace{-1mm}+\hspace{-1mm}O\big(1\big)}
{{C}_{41}(2-2d)m^{1+2d}\hspace{-1mm}+\hspace{-1mm}O\big(m^{2d-1}\big)+\hspace{-1mm}\frac{c_{1}}{c_{0}} {C'}_{42}(3)\log(m)\hspace{-1mm}+\hspace{-1mm}O\big(1\big)}\\
&\hspace{-3mm}=&\hspace{-3mm}1\hspace{-1mm}-\hspace{-1mm}\frac{2}{{C}_{41}(2-2d)}\Big[{C}_{61}(2-2d)\hspace{-1mm}+\hspace{-1mm}\frac{c_{1}}{c_{0}}{C'}_{62}(3)m^{-\beta}\log(m)\Big]\Big[1\hspace{-1mm}-\hspace{-1mm}\frac{c_{1}}{c_{0}}\frac{{C'}_{42}(3)}{{C}_{41}(2-2d)}m^{-\beta}\log(m)\Big]\hspace{-1mm}+\hspace{-1mm}O\big(m^{-2}+m^{-2d-1}\big)\\
&\hspace{-3mm}=&\hspace{-3mm}1\hspace{-1mm}-\hspace{-1mm}\frac{2{C}_{61}(2-2d)}{{C}_{41}(2-2d)}\hspace{-1mm}+\hspace{-1mm}2\frac{c_{1}}{c_{0}}\Big[\frac{{C}_{61}(2-2d){C'}_{42}(3)}{{C}_{41}(2-2d){C}_{41}(2-2d)}\hspace{-1mm}-\hspace{-1mm}\frac{{C'}_{62}(3)}{{C}_{41}(2-2d)}\Big]m^{-\beta}\log(m)\hspace{-1mm}+\hspace{-1mm}O\big(m^{-2}\big).
\end{eqnarray*}
As a consequence,
\begin{eqnarray}\label{C0C1}
\frac{R_{m}}{V_{m}^{2}}=\rho(d)~+~O\big(m^{-2d-1}~\log(m)+m^{-2} \big)\quad (m\to \infty),\quad \mbox{with $2<\beta=2d+1<4$.}
\end{eqnarray}
%with $C'_1(d,\beta)<0$ for all $d\in (-0.5,0.5)$ and $\beta\in  (0,1+2d)$ and $C'_2(\beta)<0$ for all $0<\beta<2$. \\
{\bf Step 2:} A Taylor expansion of $\Lambda(\cdot)$ around $\rho(d)$ provides:
\begin{eqnarray*}
\Lambda\Big(\frac{R_{m}}{V_{m}^{2}}\Big) \simeq \Lambda\big(\rho(d)\big)+\Big[\frac{\partial\Lambda}{\partial\rho}\Big](\rho(d))\Big(\frac{R_{m}}{V_{m}^{2}}-\rho(d)\Big)+\frac{1}{2} \, \Big[\frac{\partial^{2}\Lambda}{\partial\rho^{2}}\Big](\rho(d))\Big(\frac{R_{m}}{V_{m}^{2}}-\rho(d)\Big)^2.
\end{eqnarray*}
Note that numerical experiments show that $\displaystyle \Big[\frac{\partial\Lambda}{\partial \rho}\Big](\rho)>0.2$ for any $\rho\in (-1,1)$.
As a consequence, using the previous expansions of $R_{m}/V_{m}^{2}$ obtained in Step 1 and since $\E \big [IR_N(m)\big ]=\Lambda\big(R_{m}/V_{m}^{2}\big)$, then for all $0<\beta\leq 2$:
\begin{eqnarray*}
\E \big [IR_N(m)\big ]=\Lambda_{0}(d)+\left \{ \begin{array}{ll} c_1 \, {C'}_1(d,\beta)\,m^{-\beta}+O\big (m^{-2} +m^{-2\beta} \big )& \mbox{if}~~~\beta<2d-1 \\
 c_1 \, {C'}_2(d,\beta)\,m^{-\beta}+O\big (m^{-2}+m^{-2-\beta}\log m +m^{-2\beta} \big )& \mbox{if}~~~\beta=2d-1 \\
  c_1 \,{C'}_3(d,\beta)\,m^{-\beta}+O\big (m^{-\beta-\epsilon}+m^{-2d-1}\log m +m^{-2\beta} \big )& \mbox{if}~~~2d-1<\beta<2d+1 \\
O\big(m^{-2d-1}\log m+m^{-2}\big)& \mbox{if}~~~\beta=1+2d \\

\end{array}\right .
\end{eqnarray*}
with $C'_\ell(d,\beta)=\Big [\frac{\partial\Lambda}{\partial\rho}\Big](\rho(d)) \, C_\ell(2-2d,\beta)$ for $\ell=1,2,3$ and $ C_\ell$ defined in (\ref{Ctild1}), (\ref{Ctild2}) and (\ref{Ctild3}).
\end{proof}

\begin{proof}[Proof of Theorem \ref{cltnada}] Using Property \ref{devEIR}, if $m \simeq C\, N^\alpha$ with $C>0$ and $(1+2\beta)^{-1}  <\alpha<1$ then $\sqrt{N/m}\, \big (\E \big [IR_N(m)\big ] -\Lambda_{0}(d)\big ) \limiteN 0$ and it implies that the multidimensional CLT (\ref{TLC1}) can be replaced by
\begin{eqnarray}\label{TLC1bis}
\sqrt{\frac{N}{m}}\Big (IR_N(m_j)-\Lambda_{0}(d)\Big )_{1\leq j \leq p}\limiteloiN {\cal N}(0, \Gamma_p(d)).
\end{eqnarray}
It remains to apply the Delta-method with the function $\Lambda_0^{-1}$ to CLT (\ref{TLC1bis}). This is possible since the function $d \to \Lambda_0(d)$ is an increasing function such that $\Lambda_0'(d)>0$ and $\big (\Lambda_0^{-1})'(\Lambda_0(d))=1/\Lambda'_0(d)>0$ for all $d\in (-0.5,1.5)$. It achieves the proof of Theorem \ref{cltnada}.
\end{proof}

\begin{proof}[Proof of Proposition \ref{hatalpha}]
See Bardet and Dola (2012).
\end{proof}

\begin{proof}[Proof of Theorem \ref{tildeD}]
See Bardet and Dola (2012). \end{proof}

\section*{Appendix}
We first recall usual equalities frequently used in the sequel:
\begin{lemma}\label{Taqqup31}
For all $\lambda >0$
\begin{enumerate}
\item For $a\in(0,2 )$, $\displaystyle\frac{2}{|\lambda|^{a-1}}\int_{0}^{\infty}\frac{\sin(\lambda x)}{x^{a}}dx=\frac{4a}{2^{a}|\lambda|^{a}}\int_{0}^{\infty}\frac{\sin^{2}(\lambda x)}{x^{a+1}}dx=\frac{~\pi~}{\Gamma(a)\sin(\frac{a \pi}{2})}$;
\item For $b\in(-1,1 )$, $\displaystyle \frac{1}{2^{1-b}-1}\int_{0}^{\infty}\frac{\sin^{4}(\lambda x)}{x^{4-b}}dx=\frac{16}{-15+6\cdot2^{3-b}-3^{3-b}}\times \int_{0}^{\infty}\frac{\sin^{6}(\lambda x)}{x^{4-b}}dx=\frac{~2^{3-b}|\lambda|^{3-b}~\pi~}{4~\Gamma(4-b)\sin(\frac{(1-b) \pi}{2})}$;
\item For $b\in(1,3 )$, $\displaystyle \frac{1}{1-2^{1-b}}\int_{0}^{\infty}\frac{\sin^{4}(\lambda x)}{x^{4-b}}dx=\frac{16}{15-6\cdot2^{3-b}+3^{3-b}}\times \int_{0}^{\infty}\frac{\sin^{6}(\lambda x)}{x^{4-b}}dx=\frac{~2^{3-b}|\lambda|^{3-b}~\pi~}{4~\Gamma(4-b)\sin(\frac{(3-b) \pi}{2})}$.
    \end{enumerate}
\end{lemma}
\begin{proof}
These equations are given or deduced (using decompositions of $\sin^j(\cdot)$ and integration by parts) from (see Doukhan {\it et al.},~p. 31).\\
\end{proof}
\begin{lemma}\label{lemma46}
For $j=4,6$, denote
\begin{equation}\label{Jjdef}
J_j(a,m):=\int_{0}^{\pi}x^{a}\frac{\sin^{j}(\frac{mx}{2})}{\sin^{4}(\frac{x}{2})}dx.
\end{equation}
Then, we have the following expansions when $m \to \infty$:
\begin{eqnarray}\label{Jj}
 J_{j}(a,m)=\left \{ \begin{array}{ll}
                      {C}_{j1}(a)\,m^{3-a}+O\big(m^{1-a}\big) & \mbox{if $-1<a<1$}\\
{C}'_{j1}(1)\,m^{3-a}+O\big (\log(m)\big) & \mbox{if $a=1$} \\
{C}'_{j1}(a)\,m^{3-a}+O\big (1\big)& \mbox{if $1<a<3$} \\
{C}'_{j2}(3)\,\log(m)+O\big (1\big)& \mbox{if $a=3$} \\
{C''}_{j1}(a)+O\big(m^{-((a-3)\wedge2)})& \mbox{if $a>3$}
                     \end{array}
\right .
\end{eqnarray}
with the following real constants (which do not vanish for any $a$ on the corresponding set):
\begin{eqnarray*}\label{AllConstantTild}
&\bullet & {C}_{41}(a):=\frac{4~~\pi(1-\frac{2^{3-a}}{4})}{(3-a)\Gamma(3-a)\sin(\frac{(3-a)\pi}{2})}\quad  \mbox{and}\quad {C}_{61}(a):=\frac{\pi(15-6\cdot2^{3-a}~+3^{3-a})}{4(3-a)\Gamma(3-a)\sin(\frac{(3-a)\pi}{2})}\\
&\bullet & {C}'_{41}(a):=\Big(\frac{6}{3-a}\textbf{1}_{\{1\leq a<3\}}+16\int_{0}^{1}\frac{\sin^{4}(\frac{y}{2})}{y^{4-a}}dy+2\int_{1}^{\infty}\frac{1}{y^{4-a}}\Big(-4\cos(y)+\cos(2y)\Big)dy\Big)\\
&& \mbox{and} \quad {C}'_{61}(a):=\Big[16\int_{0}^{1}\frac{\sin^{6}(\frac{y}{2})}{y^{4-a}}dy+\frac{5}{3-a}\textbf{1}_{\{1\leq a<3\}}+\frac{1}{2}\int_{1}^{\infty}\frac{1}{y^{4-a}}\Big(-15\cos(y)+6\cos(2y)-\cos(3y)\Big)dy\Big]\\
&\bullet& {C}'_{42}(a):=\Big(6\cdot\textbf{1}_{\{a=3\}}+\textbf{1}_{\{a=1\}}\Big)\quad  \mbox{and}\quad
{C}'_{62}(a):=\Big(5\cdot\textbf{1}_{\{a=3\}}+\frac{5}{6}\cdot\textbf{1}_{\{a=1\}}\Big)\\
&\bullet& {C}''_{41}(a):=\frac{3}{8}\int_{0}^{\pi}\frac{x^{a}}{\sin^{4}(\frac{x}{2})}dx\quad  \mbox{and}\quad
{C}''_{61}(a):=\frac{5}{16}\int_{0}^{\pi}\frac{x^{a}}{\sin^{4}(\frac{x}{2})}dx.
\end{eqnarray*}
\end{lemma}
\begin{proof} The proof of these expansions follows the steps than those of Lemma 5.1 in Bardet and Dola (2012). Hence we write for $j=4,6$,
\begin{eqnarray}\label{Jjam}
J_j(a,m)&=&\widetilde{J}_{j}(a,m)+ \int_{0}^{\pi}x^{a}\sin^{j}(\frac{mx}{2})\frac{1}{(\frac{x}{2})^{4}}dx+ \int_{0}^{\pi}x^{a}\sin^{j}(\frac{mx}{2})\frac{2}{3}\frac{1}{(\frac{x}{2})^{2}}dx
\end{eqnarray}
with
\begin{eqnarray*}
\widetilde{J}_{j}(a,m)&:=&\int_{0}^{\pi}x^{a}\sin^{j}(\frac{mx}{2})\big(\frac{1}{\sin^{4}(\frac{x}{2})}-\frac{1}{(\frac{x}{2})^{4}}-\frac{2}{3}\frac{1}{(\frac{x}{2})^{2}}\big)dx.
\end{eqnarray*}
The expansions when $m\to \infty$ of both the right hand sided integrals in (\ref{Jjam}) are obtained from Lemma \ref{Taqqup31}. It remains to obtain the expansion of $\widetilde{J}_{j}(a,m)$. Then, using classical trigonometric and Taylor expansions:
\begin{eqnarray*}
\sin^{4}(\frac{y}{2})&=& \frac 1 8 \big ( 3-4\cos(y)+\cos(2y) \big )\quad \mbox{and}\quad \frac{1}{\sin^{4}(y)}-\frac{1}{y^{4}}-\frac{2}{3}\frac{1}{y^{2}}\sim \frac{11}{45}\quad (y \to 0)\\
\sin^{6}(\frac{y}{2})&=&\frac 1 {32} \big (10-15\cos(y)+6\cos(2y)-\cos(3y) \big )\quad \mbox{and}\quad \frac{1}{y^{5}}+\frac{1}{3}\frac{1}{y^{3}}-\frac{\cos(y)}{\sin^{5}(y)}\sim \frac{31}{945}\, y\quad (y \to 0),
\end{eqnarray*}
the expansions of $\widetilde{J}_{j}(a,m)$ can be obtained.

Numerical experiments show that ${C}''_{41}(a)\neq0$, ${C}''_{61}(a)\neq0$, ${C}''_{42}(a)\neq0$ and ${C}''_{62}(a)\neq0$.
\end{proof}

%\bibliographystyle{amsalpha}
%\printindex
%\bibliographystyle{acm}
\bibliographystyle{amsalpha}

\end{document}